\documentclass{article}

\usepackage{geometry}
\usepackage[all]{xy}
\usepackage{amsmath}
\usepackage{amssymb}
\usepackage{amsthm}
\usepackage{hyperref}
\title{\normalsize \bf FINE SHAPE OF METRIZABLE SPACES AS A LEFT FRACTION LOCALIZATION}
\author{\footnotesize VLADISLAV ZEMLYANOY \\
  \footnotesize {\it Doctoral School in Mathematics, Higher School of Economics (HSE University)} \\
  \footnotesize {\it Moscow, Russia} \\
  \footnotesize {\it tagambit@yandex.ru} \\
  \footnotesize {\it ORCID} 0000-0001-7450-5325
}
\date{}

\begin{document}
\def \R {\mathbb{R}}
\def \N {\mathbb{N}}
\def \i {\imath}
\def \j {\jmath}

\setlength{\parskip}{0.25em}

\newcounter{count}[section]
\renewcommand{\thecount}{\arabic{section}.\arabic{count}}

\theoremstyle{plain}
\newtheorem{theorem}[count]{Theorem}
\newtheorem*{theorem*}{Theorem}
\newtheorem{lemma}[count]{Lemma}
\newtheorem{proposition}[count]{Proposition}
\newtheorem{corollary}[count]{Corollary}

\theoremstyle{definition}
\newtheorem{definition}[count]{Definition}
\newtheorem*{convention}{Convention}
\newtheorem*{notation}{Notation}
\newtheorem{example}[count]{Example}

\theoremstyle{remark}
\newtheorem{remark}[count]{Remark}

\maketitle

\begin{abstract}
\noindent The strong shape category of compact metrizable spaces 
(compacta) is very well-studied; extending it to noncompact spaces, 
however, introduces computational complexity that makes it hard to work 
with. The fine shape category, as defined by Melikhov, seems to 
hold promise in terms of both applicability and simplicity: it is a 
different extension of compact strong shape to a generalized homotopy 
theory of metrizable spaces that is compatible with both \v{C}ech 
cohomology and Steenrod-Sitnikov homology, and its definition lends 
itself to straightforward proofs. Further research seems to be in order. 
One goal to have in mind is to show the fine shape category to be a 
homotopy category in Quillen's sense, which implies representation as a 
localization. But the strong shape of compacta was shown to be a left 
fraction localization in several ways; we extend the representation 
given by Cathey to fine shape, introducing the notion of FDR-embeddings 
to extend Cathey's SSDR-maps. In the process, we also introduce what we 
call the mapping cylinder of an approaching map; such a construction has 
been defined by Ferry and elaborated on by Mrozik in the compact case, 
yet it seems the direct extension on noncompact spaces is not possible. 
Thus we resort to a somewhat different definition.
\end{abstract}

\section{Introduction}
\label{sec_intro}

Strong shape theory has been introduced as early as 1944 by 
Christie~\cite{Christie1944}. Many definitions of the same have been 
later given independently for compacta, that is, compact metrizable 
spaces~\cite{Bauer1978, Kodama1979, Calder1981, Cathey1981, DydakSegal1981}, 
all equivalent on these, and only differing in representation; these 
mostly arose as modifications of Borsuk's shape theory after its 
publication~\cite{Borsuk1968}. The sole resulting category over compacta 
is well-established, having been extensively studied and used to solve 
various problems. There is also a known strong shape category of all 
topological spaces~\cite{Batanin1997,Mardesic2000}, which is quite 
complex in both definition and computations; thus it is much less 
researched or applied.

A different notion has been suggested by Melikhov~\cite{Melikhov2022F}, 
called fine shape. This name originally was used by Kodama and 
Ono~\cite{Kodama1979} for compacta only, whereas~\cite{Melikhov2022F} 
extends it to all metrizable spaces. This fine shape still coincides 
with strong shape on compacta. However, fine shape has a far simpler 
definition that the noncompact strong shape; thus, is should be easier 
to prove meaningful results for it. In particular, invariants of fine 
shape include (see~\cite[Corollary~1.3]{Melikhov2022F}) both 
\v{C}ech cohomology and Steenrod-Sitnikov homology (the latter being 
defined as the direct limit of Steenrod homology of compacta; 
see~\cite[footnote~3]{Melikhov2022F}). These two theories seem to be a 
good choice of universal homology and cohomology pair for metrizable 
spaces (a survey to that effect can be found in~\cite{Sklyarenko1979}). 
Thus this fine shape theory promises to be a working (generalized) 
homotopy theory for metrizable spaces complementing \v{C}ech cohomology 
and Steenrod-Sitnikov homology.

The construction of this fine shape that we shall use relies on the 
notion of an approaching map, originally going as far back as Quigley's 
work~\cite{Quigley1973}. The version we use is somewhat more general 
than the one often used in the past for compacta:

\begin{definition}
Let $X$ be a closed subspace of a metrizable topological space $M$. We 
say that $X$ is {\it homotopy negligible} in $M$ to mean that there is a 
homotopy $H\colon M \times [0,1] \to M$ such that $H_{0} = id_{M}$ (the 
identity map of $M$) and $H(M \times (0,1]) \subseteq M\setminus X$ (so 
the image of $M$ under $H_{t}$ does not cross $X$ for $t>0$).

Now assume also $Y$ is closed and homotopy negligible in $N$ in the same 
way, and that $\phi\colon M\setminus X \to N\setminus Y$ is a continuous 
map. We say that $\phi$ is an {\it $X-Y$-approaching map} to mean that 
for every sequence $\{m_{i}\} \subset M\setminus X$ converging to a 
point of $X$, the sequence $\{\phi(m_{i})\}$ has a subsequence 
converging to a point of $Y$ (in other terms, $\{\phi(m_{i})\}$ has an 
accumulation point in $Y$).
\end{definition}

There are several other equivalent 
definitions~\cite[Theorem~1]{Melikhov2022F}. It should be also noted 
that if we restrict ourselves to compact metrizable spaces, this is 
equivalent to saying that $\phi$ is proper (i.e., inverse images of 
compact subsets of $N\setminus Y$ are compact). This provides a way of 
showing that fine shape restricts to strong shape on compacta.

Now fine shape itself can be explained as follows:

\begin{definition}
Given two metrizable spaces $X$ and $Y$ and any absolute retracts (ARs) 
$M$ and $N$ containing $X$ and $Y$ respectively as closed homotopy 
negligible subsets, the set of {\it fine shape classes} $[X,Y]_{fSh}$ is 
the set of $X-Y$-approaching maps 
$\phi\colon M\setminus X \to N\setminus Y$ up to an $X-Y$-approaching 
homotopy (such a homotopy is an $(X \times [0,1])-Y$-approaching map). 
It turns out that the set $[X,Y]_{fSh}$ is independent of the choice of 
$M$ and $N$.

The {\it fine shape category} $\mathrm{fSh(M)}$ has all metrizable 
spaces as objects, and fine shape classes as morphisms.
\end{definition}

It is not hard to prove that every homotopy class of oridnary continuous 
maps $[f] \in [X,Y]$ determines a unique fine shape class 
$[f]_{fSh} \in [X,Y]_{fSh}$. Fine shape is in general weaker; in 
particular, a map $f$ that is not a homotopy equivalence can induce an 
isomorphism in fine shape (such a map is called a 
{\it fine shape equivalence}).

One can hope to extract further use and applications out of this version 
of shape theory for noncompact metrizable spaces where it is hard to do 
so for noncompact strong shape. A simple test of this is to take a 
result that has been proven for compact strong shape, and extend it to 
fine shape of all metrizable spaces. Yet it should also be possible to 
prove results for fine shape that do not restrict to compacta at all. In 
particular, we expect to be able to prove that the fine shape category 
is in fact a homotopy category in the sense of 
Quillen~\cite[Chapter~I, section~1, Definition~6]{Quillen1967} by 
constructing a corresponding model structure on the category 
$\mathrm{M}$ of metrizable spaces and continuous maps. An equivalent of 
this has never been constructed for the strong shape of compacta; we do 
not expect it to be possible either, because some of the spaces arising 
there (the path spaces in particular) are practically bound to be 
noncompact even when starting with compacta (in fact, Cathey runs into 
this problem in~\cite[Theorem~(2.5)]{Cathey1981}; the space $|X|$ there 
is in general not compact). This search for a model structure also 
raises the question of representing fine shape morphisms by usual maps 
of spaces via the localization used in Quillen's definition.

Among others, Cathey~\cite{Cathey1981} defines the strong shape category 
of compacta as a left fraction localization (of $\mathrm{hCM}$, the 
category of compacta and homotopy classes of continuous maps); it is 
thus known that this description of compact strong shape is equivalent 
to the others. Also works by Calder and Hastings~\cite{Calder1981} and 
by Mrozik~\cite{Mrozik1990} construct localizations by different 
morphism classes (of the same category), and show those still to be 
equivalent to the same strong shape category of compacta.

The goal of the present work is to extend Cathey's definition to fine 
shape, representing the fine shape category of~\cite{Melikhov2022F} as a 
left fraction localization. This gives us a simple representation of 
fine shape by (homotopy classes of) usual maps:

\begin{theorem}~(Corollary~\ref{fSh_represent})
For any fine shape class $[\phi] \in [X,Y]_{fSh}$, there exist a space 
$Z$ along with maps (which can be chosen to be closed embeddings) 
$u\colon X \to Z$ and $i\colon Y \to Z$ such that $i$ is a fine shape 
equivalence and $[i]_{fSh}^{-1}\circ[u]_{fSh} = [\phi]$.
\end{theorem}

Moreover, if fine shape will be proven to be a homotopy category in 
Quillen's sense (which we expect to do), we will be able to represent 
its morphisms with left fractions of homotopy classes alone, as opposed 
to the general case (which uses equivalence classes of finite chains of 
morphisms and inversions of morphisms of the original category); this 
will provide a convenient way of working with $\mathrm{fSh(M)}$ as a 
homotopy category. In fact, due to the universal property of left 
fraction localization 
(stated for reference in Proposition~\ref{left_universal}), the homotopy 
category functor $\mathrm{M} \to \mathrm{fSh(M)}$, if it exists, must 
factorize through $\mathrm{hM}$. We structure our work as follows.

Section~\ref{sec_prelim} contains all previously established results we 
shall make use of; it is by necessity quite large, to ensure that all 
inobvious definitions and statements are explained and supported by 
references. In section~\ref{sec_fdr}, we introduce and explore the 
notion of an FDR-embedding, extending Cathey's SSDR-maps to noncompact 
metrizable spaces. Same as for these, we have a simple description 
(Theorem~\ref{WC_is_fdr}): an FDR-embedding is precisely a closed 
embedding that is a fine shape equivalence. In spirit 
of~\cite{Cathey1981}, in section~\ref{sec_fdr_localization} we consider 
the class $[FDR]$ of homotopy classes of FDR-embeddings in $\mathrm{hM}$, 
the category of metrizable spaces and homotopy classes of continuous 
maps; we show (Corollary~\ref{left_exists}) that there exists a category 
$[FDR]\backslash\mathrm{hM}$, the left fraction localization of 
$\mathrm{hM}$ at $[FDR]$ (the same cannot be done without resorting to 
homotopy classes). Section~\ref{sec_functor_S} defines the functor $S$ 
from this localization to the fine shape category $\mathrm{fSh(M)}$. 
Construction of the inverse functor $T$, as well as the proof of the 
inversion, are postponed until section~\ref{sec_functor_T}; there we show 
(Theorem~\ref{ST_inverse}) the two categories to be isomorphic. Thus the 
final result of the present work can be condensed to the following

\begin{theorem}~(Corollary~\ref{left_exists},~Theorem~\ref{ST_inverse})
There exists a well-defined category of left fractions 
$[FDR]\backslash\mathrm{hM}$, isomorphic to the fine shape category 
$\mathrm{fSh(M)}$.
\end{theorem}

One additional notable fact arises during the proof:

\begin{theorem}~(Corollary~\ref{maps_fSh_homotopic})
Given two maps $f,g\colon X \to Y$ such that $[f]_{fSh} = [g]_{fSh}$, 
there exist a space $Z$ along with a map $h\colon Y \to Z$ such that $h$ 
is a closed embedding and a fine shape equivalence, and $hf$ is 
homotopic to $hg$.
\end{theorem}

Section~\ref{sec_AC_FC} introduces a construction we use in the process. 
Specifically, we extend the concept of mapping cylinder to 
approaching maps between metrizable spaces; for compacta, this has been 
done explicitly by Mrozik~\cite{Mrozik1990}, refering an earlier 
construction of Ferry~\cite{Ferry1980}, yet it seems that a 
modification is required for the noncompact case. This modification is 
what section~\ref{sec_AC_FC} is devoted to; the appendix at the end 
elaborates on why the modification is needed at all. Our description is 
unfortunately quite cumbersome (if not complex), though we do offer some 
explanation of what it actually represents (see 
Remark~\ref{FC_blow_up}). What we prove for this cylinder construction 
suffices for our goals, yet there are some unresolved questions there; 
in particular (see~Remark~\ref{AC_universal} and the appendix), the 
question of the universal property echoing that of the usual mapping 
cylinder. We do not dwell on this question here, but it may still be of 
some interest.

\section{Preliminaries}
\label{sec_prelim}

In this section, various notions and results are collated for 
reference. All of these are either previously established or trivial 
(often both), though not all may be widely known.

\subsection{Spaces}
\label{def_spaces}

We work exclusively with metrizable topological spaces and 
(topologically) continuous maps between them, so we adopt the following

\begin{convention}
By a space, we shall always mean a metrizable topological space, unless 
specified otherwise. By a map between spaces, we shall always mean a 
continuous map.
\end{convention}

\begin{definition}
By $\mathrm{M}$ we denote the category of metrizable topological spaces 
and continuous maps. By $\mathrm{hM}$ we denote the category of 
metrizable topological spaces and homotopy classes of continuous maps. 
\end{definition}

\begin{notation}
We denote the homotopy class of a map $f$ by $[f]$, and we write 
$f \simeq g$ to mean that maps $f$ and $g$ are homotopic. We denote the 
set of all homotopy classes from $X$ to $Y$ by $[X,Y]$. For every space 
$X$, we denote the identity map of $X$ by $id_{X}$. Restriction of a map 
$f\colon X \to Y$ to a subspace $A \subseteq X$, we denote by $f|_{A}$.
\end{notation}

\subsection{Metrizable joins and metrizable mapping cylinders}
\label{def_metric}

As we restrict ourselves to metrizable spaces, we will need a number of 
constructions that provide metrizable analogues to well-known 
topological objects. This subsection is fully based 
on~\cite[Chapter II]{Melikhov2022T}.

%It is known that a topology on a given set is not uniquely determined by 
%sequence convergence (which is why net convergence is sometimes 
%defined), but a metrizable topolgy is. That is, any two metrizable 
%topologies on a set coincide if and only if they have the same sequence 
%convergence. Thus, in particular, 

\begin{definition}\label{def_join}
Given two spaces $M$ and $N$ whose topologies are generated by some 
metrics $d_{M}$ and $d_{N}$, both bounded by unity, the 
{\it metrizable join} $M \star N$ is the space that has the same 
underlying set as the usual (topological) join --- 
$(M \times [-1,1] \times N)/\sim$, where $(m,1,n) \sim (m,1,n')$ and 
$(m,-1,n) \sim (m',-1,n)$ for all $m,m' \in M$ and $n,n' \in N$ --- with 
the topology given by the metric~\cite[before Remark~7.26]{Melikhov2022T}
$$
d((m,s,n),(m',s',n')) := \min\left\{
\begin{array}{lr}
  d_{M}(m,m') + |s - s'| + d_{N}(n,n'), \\
  d_{M}(m,m') + |1 - s| + |1 - s'|, \\
  |s + 1| + |s' + 1| + d_{N}(n,n'), \\
  4 - |s - s'| \\
\end{array}
\right\}.
$$
\end{definition}

\begin{remark}\label{join_remarks}
(1)As per the reference, the topology of $M \star N$ is independent of 
the choice of $d_{M}$ and $d_{N}$.

(2)Often (but not always) this space simply coincides with the 
topological join --- for example, whenever $M$ and $N$ are compact.

(3)Since we work with metrizable spaces, we reiterate that the notation 
$M \star N$ shall always refer to the metrizable join, rather that the 
topological one.

(4)We note explicitly that in $M \star N$, 
$(m_{k},s_{k},n_{k}) \to (m,1) = m$ if and only if $m_{k} \to m$ and 
$s_{k} \to 1$, $(m_{k},s_{k},n_{k}) \to (-1,n) = n$ if and only if 
$n_{k} \to n$ and $s_{k} \to -1$, and for $s \in (-1,1)$, 
$(m_{k},s_{k},n_{k}) \to (m,s,n)$ if and only if $m_{k} \to m$, 
$s_{k} \to s$, and $n_{k} \to n$. Thus metrizable join has not only the 
same underlying set, but also the same sequence convergence as the 
topological one.

(5)It is easy to see that for a given topology on a set, there can be at 
most one metrizable topology on the same set with the same sequence 
convergence. Therefore, metrizable join is uniquely defined as the 
metrizable space with the same underlying set and sequence convergence 
as the topological join, and the explicit metric shows that it does 
exist for any two metrizable spaces.

(6)As usual, we have standard embeddings of $M$ and $N$ into their join. 
\end{remark}

\begin{notation}
A point $m \in M$ in the join $M \star N$ shall be denoted by $(m,1)$, 
or by $(m,1,n)$ for any $n \in N$, or simply by $m$. A point $n \in N$, 
similarly, by $(-1,n)$, $(m,-1,n)$, or simply $n$. Any other point of 
$M \star N$ shall be denoted by $(m,s,n) \in M \times (-1,1) \times N$. 
\end{notation}

From this join, we can construct the metrizable mapping cylinder. We do 
this by embedding the latter in the former (the same can be done in the 
usual topological case):

\begin{definition}\label{def_MC}
Given spaces $M$ and $N$, and a map $f\colon M \to N$, we define the 
{\it metrizable mapping cylinder} $MC(f)$ to be the unique metrizable 
space having the same underlying set (which we take to be 
$M \times (0,1] \cup N$) and sequence convergence as the topological 
cylinder of $f$. $MC(f)$ can be constructed by inserting the underlying 
set into $M \star N$ with $N \ni n \mapsto n$ and 
$M \times (0,1] \ni (m,s) \mapsto (m,2s-1,f(m))$, and taking the 
subspace topology.
\end{definition}

\begin{remark}\label{MC_remarks}
(1)This explicit construction induces the product topology on 
$M \times (0,1]$, and the original topology on $N$. In particular, we 
have two standard embeddings, which we denote, within this remark, by 
$u\colon M = M \times \{1\} \to MC(f)$ and $i\colon N \to MC(f)$.

(2)$MC(f)$ strong deformation retracts onto its base $N$, same as in the 
topological case. What is more, this provides a homotopy (from 
$M \times [0,1]$ to $MC(f)$) between $u$ and $if$.

(3)Yet further, there is a strong deformation retraction of 
$MC(f) \times [0,1]$ onto $u(M) \times [0,1] \cup MC(f) \times \{0\}$: 
first take a strong deformation retraction of $[0,1] \times [0,1]$ onto 
$\{1\} \times [0,1] \cup [0,1] \times \{0\}$ in which 
$\{0\} \times [0,1]$ retracts along itself onto $\{0\} \times \{0\}$; 
then, for each $m \in M$, apply this strong deformation retraction to 
the embedding 
$r_{m} \times id_{[0,1]}\colon [0,1] \times [0,1] \to MC(f) \times [0,1]$, 
where $r_{m}(s) := (m,s)$ for $s > 0$ and $r_{m}(0) := f(m)$.
\end{remark}

We shall also extend the cylinder construction a little:

\begin{definition}
Assume that $L$ is a subset of $M$, and that we have a map 
$f\colon L \to N$. Then we shall say that $f$ is a {\it partial map} 
from $M$ to $N$, and define the metrizable mapping cylinder of $f$ as a 
partial map, denoted $MC_{M}(f)$, by taking the set 
$M \times \{1\} \cup L \times (0,1] \cup N$ and inserting it into 
$M \star N$ such that $M$ is embedded in the standard way, and 
$L \times (0,1] \cup N$ is embedded as $MC(f)$ into $L \star N$, which 
is then embedded into $M \star N$. In other terms, 
$MC_{M}(f) := M \cup_{L \times \{1\}} MC(f)$.
\end{definition}

Finally we shall make use of one more construction~\cite[Subsection~7.J]{Melikhov2022T}:

\begin{definition}
Assume a map $f\colon M \to N$ and a closed subset $X$ of $M$ such that 
the restriction $f|_{X}$ is perfect (i.e., closed and has compact 
inverse image of every point). Then we can take $MC(f)$ and a continuous 
function $h\colon M \to [0,1]$ such that $h^{-1}(0) = X$. By the 
{\it relative metrizable mapping cylinder of $f$ by $X$} we shall mean 
the subspace $MC(f|X) := \{(m,s) \in MC(f) \mid s \leq h(m)\} \cup N$.
\end{definition}

\begin{remark}
(1)It is easy to see that $MC(f|X)$ is a strong deformation retract of 
$MC(f)$, and that its homeomorphism class does not depend on the choice 
of the function $h$ (this uses the fact that $f|_{X}$ is perfect).

(2)Moreover, $MC(f|X)$ contains $N$ (as a closed subset) and 
$M\setminus X = \{(m,s) \mid s = h(m) > 0\}$. If $f$ is injective on 
$X$, then $MC(f|X)$ even contains $M$ as a closed subset.
\end{remark}

\subsection{Approaching maps}
\label{def_appr}

Following~\cite{Melikhov2022F}, we construct fine shape using 
approaching maps, which we define here.

\begin{definition}
Let $X$ be a closed subset of a space $M$. We say that $X$ is 
{\it homotopy negligible} in $M$ to mean that there exists a homotopy 
$H\colon M \times [0,1] \to M$ such that $H_{0} = id_{M}$ and 
$H(M \times (0,1]) \subseteq M\setminus X$. In other words, there is a 
deformation of $M$ into itself that never crosses $X$ except at the 
initial (identity) map.
\end{definition}

\begin{definition}
Let $X$ and $Y$ be closed homotopy negligible in spaces $M$ and $N$ 
respectively, and let $\phi\colon M\setminus X \to N\setminus Y$ be a 
map (continuous on its domain, as per our convention). We say that 
$\phi$ is $X-Y$-{\it approaching} to mean that for any 
sequence $\{m_{i}\} \subset M\setminus X$ converging (in $M$) to a point 
of $X$, the sequence $\{\phi(m_{i})\} \subset N\setminus Y$ contains a 
subsequence that converges (in $N$) to a point of $Y$. Equivalently, 
$\phi$ being $X-Y$-approaching means that whenever a sequence in 
$M\setminus X$ has an accumulation point in $X$, the sequence's image in 
$N\setminus Y$ has an accumulation point in $Y$.
\end{definition}

It is clear that a composition of approaching maps is an approaching 
map, and that the identity map $id_{M\setminus X}$ is $X-X$-approaching. 
Therefore we have no trouble with the following

\begin{definition}
The {\it approaching category} $\mathrm{MAppr}$ is defined as follows: 
its objects are pairs of spaces $(M,X)$ with $X$ closed homotopy 
negligible in $M$, and a morphism from $(M,X)$ to $(N,Y)$ is an 
$X-Y$-approaching map $\phi\colon M\setminus X \to N\setminus Y$.
\end{definition}

For the rest of the present work, we adopt the following

\begin{convention}
Whenever we speak of an approaching map 
$\phi\colon M\setminus X \to N\setminus Y$, we mean that $(M,X)$ and 
$(N,Y)$ are objects of $\mathrm{MAppr}$, and $\phi$ is a morphism 
between those --- that is, $\phi$ is both continuous on $M\setminus X$ 
and $X-Y$-approaching.
\end{convention}

Homotopies of approaching maps are readily defined:

\begin{definition}
Given two approaching maps 
$\phi,\psi\colon M\setminus X \to N\setminus Y$, by an {\it approaching 
homotopy} between $\phi$ and $\psi$ we mean an approaching map 
$\Phi\colon (M \times [0,1])\setminus (X \times [0,1]) \to N\setminus Y$ 
such that $\Phi_{0} = \phi$ and $\Phi_{1} = \psi$. Whenever such 
approaching homotopy exists, we say that $\phi$ and $\psi$ are 
{\it approaching homotopic}, which is an equivalence relation. Finally, 
we define the {\it approaching homotopy category} $\mathrm{hMAppr}$ to 
consist of the same objects as $\mathrm{MAppr}$ and approaching homotopy 
classes of approaching maps.
\end{definition}

In the present work, we make use of two concepts having to do with 
passing between an approaching map 
$\phi\colon M\setminus X \to N\setminus Y$ and a map defined on $X$, or 
a subset of $X$, taking it to $Y$:

\begin{definition}\label{def_phi_ex_set}
In some cases, an approaching map 
$\phi\colon M\setminus X \to N\setminus Y$ can be obtained by extending 
a map $f\colon X \to Y$, so that the two combine into a continuous map 
$\bar{f}\colon M \to N$ with $\bar{f}^{-1}(Y) = X$. In such case we say 
that $\phi$ {\it extends} $f$.

Conversely, for an approaching map 
$\phi\colon M\setminus X \to N\setminus Y$, define the 
{\it extension set} of $\phi$, denoted $X(\phi)$, as follows: a point 
$x \in X$ is in $X(\phi)$ if and only if there is a point $y \in Y$ such 
that for every sequence in $M\setminus X$ converging to $x$, the 
$\phi$-image of the sequence converges to $y$.
\end{definition}

The following fact is clear:

\begin{proposition}
$X(\phi)$ is the largest subset of $X$ for which there is a continuous 
map $f\colon X(\phi) \to Y$ combining with $\phi$ into a continuous map; 
$f$ is defined uniquely; for any subset $A$ of $X$, $\phi$ extends on 
$A$ by a map $f'\colon A \to Y$ if and only if $A \subseteq X(\phi)$, 
and in that case $f' = f|_{A}$.
\end{proposition}

\begin{remark}
Note that homotopy negligibility of $X$ in $M$ implies that every point 
of $X$ is the limit of some sequence in $M\setminus X$.
\end{remark}

Finally, we shall also make use of a specific kind of homotopy:

\begin{definition}\label{def_H_add}
Given any homotopy $H\colon X \times [0,1] \to Y$, we say $H$ is 
{\it additive} whenever $H_{s}\circ H_{t} = H_{\min\{s+t,1\}}$ for all 
$s,t \in [0,1]$. For a space $M$ and a closed subset $X$ of $M$, we say 
that $X$ is {\it additive homotopy negligible} in $M$ whenever there is 
an additive homotopy $H\colon M \times [0,1] \to M$ such that 
$H_{0} = id_{M}$ and $H(M \times (0,1]) \subset M\setminus X$.
\end{definition}

\subsection{Absolute retracts}
\label{def_ARs}

The second notion we use to construct fine shape, along with approaching 
maps, is that of an absolute retract. This, of course, goes back to 
the Borsuk's definition of shape using absolute neighborhood retracts; 
here, though, we will not need the latter.

\begin{definition}
By an {\it absolute retract (AR)}, we shall mean a space $M$ that is, in 
fact, an absolute extensor for all metrizable spaces in the following 
sense: given any space $X$ and a closed subset $A$ of $X$, any map 
$f\colon A \to M$ can be extended to a map $\bar{f}\colon X \to M$ such 
that $\bar{f}|_{A} = f$:
$$\xymatrix{
A \ar[r]^{f} \ar[d]_{} & M \\
X \ar@{.>}[ur]_{\bar{f}}
}$$
\end{definition}

It is known~\cite[Corollary~18.3]{Melikhov2022T} that (at least with 
respect to metrizable spaces) absolute retracts and absolute extensors 
(AEs) are exactly the same spaces; by convention, we call them all ARs.

Aside from the definition, we make use of the following facts about ARs:

\begin{proposition}~\cite[Chapter~18]{Melikhov2022T}\label{AR_props}
(1)Every contractible polyhedron is an AR, including, in particular, the 
unit interval $[0,1]$;

(2)A direct product of any countable set of ARs is an AR;

(3)If $M$ and $N$ are ARs, then so is $M \star N$, and so is $MC(f)$ 
for any $f\colon M \to N$;

(4)If $M = M_{1} \cup M_{2}$ is such that $M_{1}$ and $M_{2}$ are ARs, 
both are closed in $M$, and $M_{1} \cap M_{2}$ is an AR, then $M$ is an 
AR (therefore, $MC_{M}(f)$ is an AR for any $f\colon A \to N$ such that 
$M$, $N$, and $A$ are ARs, and $A$ is a closed subset of $M$);

(5)A retract of an AR is an AR (therefore, $MC(f|X)$ is an AR for any 
$f\colon M \to N$ such that $M$ and $N$ are ARs, $X$ is closed in $M$, 
and $f|_{X}$ is perfect);

(6)\cite[Theorem~19.3]{Melikhov2022T} If $X$ is closed and homotopy 
negligible in $M$, then $M$ is an AR if and only if $M\setminus X$ is an 
AR;

(7)For any space $X$, there exists an AR $M$ containing $X$ as a closed 
subset (therefore, $M \times [0,1]$ contains $X = X \times \{0\}$ as a 
closed additive homotopy negligible subset, homotopy given by 
$H_{t}(m,s) := (m,\min\{s+t,1\})$).
\end{proposition}

\subsection{Fine shape}
\label{def_fine_shape}

The notions of an approaching map and of an absolute retract are 
combined into the notion of fine shape based on 
the following fact~\cite[Lemma~2.13]{Melikhov2022F}:

\begin{lemma}\label{extend_map_to_fSh}
Let $X$ be a closed subset of any space $M$, and $Y$ be closed homotopy 
negligible in an AR $N$. Then

(1)Every map $f\colon X \to Y$ extends to a map $\bar{f}\colon M \to N$ 
such that $\bar{f}^{-1}(Y) = X$ (and therefore 
$\bar{f}|_{M\setminus X}$ is $X-Y$-approaching);

(2)For any homotopy $F\colon X \times [0,1] \to Y$ and any extensions 
$\bar{F_{0}},\bar{F_{1}}\colon M \to N$ of $F_{0}$ and $F_{1}$ such that 
$\bar{F_{0}}^{-1}(Y) = \bar{F_{1}}^{-1}(Y) = X$, there is an extension 
$\bar{F}\colon M \times [0,1] \to N$ such that 
$\bar{F}(-,0) = \bar{F_{0}}$, $\bar{F}(-,1) = \bar{F_{1}}$, and 
$\bar{F}^{-1}(Y) = X$.
\end{lemma}

\begin{proof}
(1)First we can extend $f$ to any map $f'\colon M \to N$, since $N$ is 
an AR. Then choose any homotopy $H\colon N \times [0,1] \to N$ such that 
$H_{0} = id_{N}$ and $H(N \times (0,1]) \subseteq N\setminus Y$, and any 
continuous function $h\colon M \to [0,1]$ such that $h^{-1}(0) = X$. 
From those we define $\bar{f}(m) := H_{h(m)}\circ f'(m)$. Then 
$\bar{f}(M\setminus X) \subseteq N\setminus Y$, and 
$\bar{f}|_{M\setminus X}$ is $X-Y$-approaching (as it actually extends 
on $X$ by a map into $Y$), as needed.

(2)$M \times \{0,1\} \cup X \times [0,1]$ is a closed subset of 
$M \times [0,1]$; we combine the maps $\bar{F}_{0}$, $\bar{F}_{1}$, and 
$F$ into a map of $M \times \{0,1\} \cup X \times [0,1]$ into $N$, then 
extend it to a map $\bar{F}'\colon M \times [0,1] \to N$. Now same as in 
1), choose any homotopy $H\colon N \times [0,1] \to N$ such that 
$H_{0} = id_{N}$ and $H(N \times (0,1]) \subseteq N\setminus Y$, and any 
continuous function $h\colon M \times [0,1] \to [0,1]$ such that 
$h^{-1}(0) = M \times \{0,1\} \cup X \times [0,1]$, and define 
$\bar{F}(m,t) := H_{h(m)}\circ\bar{F}'(m,t)$.
\end{proof}

Now fine shape is constructed from the following, which is clearly an 
equivalence relation:

\begin{definition}\label{def_fSh_class}
Let $X$ and $Y$ be any two spaces. Let $M$ and $M'$ be ARs, each 
containing $X$ as a closed homotopy negligible subset, whereas $N$ and 
$N'$ are ARs each containing $Y$ as such. For any approaching maps 
$\phi\colon M\setminus X \to N\setminus Y$ and 
$\psi\colon M'\setminus X \to N'\setminus Y$, we say that $\phi$ and 
$\psi$ {\it are of the same fine shape class} (from $X$ to $Y$) 
whenever there are some maps $\bar{id}_{X}\colon M \to M'$ and 
$\bar{id}_{Y}\colon N \to N'$, extending $id_{X}$ and $id_{Y}$, such 
that $\bar{id}_{X}^{-1}(X) = X$, $\bar{id}_{Y}^{-1}(Y) = Y$, and that 
there is an approaching homotopy (from $(M\setminus X) \times [0,1]$ to 
$N'\setminus Y$) between $\bar{id}_{Y}\circ\phi$ and 
$\psi\circ\bar{id}_{X}$, so that the following diagram commutes in 
approaching homotopy:
$$\xymatrix{
M\setminus X \ar[r]^{\phi} \ar[d]_{\bar{id}_{X}|_{M\setminus X}} & N\setminus Y \ar[d]^{\bar{id}_{Y}|_{N\setminus Y}} \\
M'\setminus X \ar[r]_{\psi} & N'\setminus Y
}$$
\end{definition}

By using Lemma~\ref{extend_map_to_fSh}(1), we can construct fine shape 
classes from ordinary maps:

\begin{definition}
Given a map $f\colon X \to Y$, the fine shape class (from $X$ to $Y$) 
{\it induced by $f$}, denoted $[f]_{fSh}$, is defined as follows: take 
any ARs $M$ and $N$ containing $X$ and $Y$ respectively as closed 
homotopy negligible subsets, extend $f$ to a map 
$\bar{f}\colon M\setminus X \to N\setminus Y$ such that 
$\bar{f}^{-1}(Y) = X$, and take the fine shape class of 
$\bar{f}|_{M\setminus X}$.
\end{definition}

There are, however, fine shape classes that are not induced by any 
ordinary maps. This is well-known already for strong shape of compacta; 
we mention a class of examples in Remark~\ref{fdr_nontrivial}.

By definition, every fine shape from $X$ to $Y$ is represented 
by some approaching map $\phi\colon M\setminus X \to N\setminus Y$ for 
some ARs $M$ and $N$ (containing $X$ and $Y$ respectively as closed 
homotopy negligible subsets). In fact, however, it can be represented 
for {\it any} such choice of ARs:

\begin{lemma}
Let $X$ and $Y$ be any two spaces, $M$ and $N$ be ARs containing $X$ and 
$Y$ respectively as closed homotopy negligible subsets, and 
$\phi\colon M\setminus X \to N\setminus Y$ be an approaching map. For 
any two other ARs $M'$ and $N'$ containing $X$ and $Y$ respectively as 
closed homotopy negligible subsets, there is an approaching map 
$\psi\colon M'\setminus X \to N'\setminus Y$ that is of the same fine 
shape class as $\phi$.
\end{lemma}

\begin{proof}
Take some maps $\bar{id}_{X}\colon M' \to M$ and 
$\bar{id}_{Y}\colon N \to N'$ extending $id_{X}$ and $id_{Y}$ 
respectively and such that $\bar{id}_{X}^{-1}(X) = X$ and 
$\bar{id}_{Y}^{-1}(Y) = Y$. Then the map 
$\bar{id}_{Y}\circ\phi\circ\bar{id}_{X}\colon M'\setminus X \to N'\setminus Y$ 
is of the same fine shape class as $\phi$.
\end{proof}

\begin{corollary}\label{fSh_composable}
Fine shape classes are composable: a fine shape class from $X$ to $Y$, 
represented by $\phi\colon M\setminus X \to N\setminus Y$, and a fine 
shape class from $Y$ to $Z$, represented by 
$\psi\colon N'\setminus Y \to L'\setminus Z$, compose through taking any 
$\phi'\colon M\setminus X \to N'\setminus Y$ of the same fine shape 
class as $\phi$ and taking the fine shape class of $\psi\circ\phi'$, or 
by taking any map $\psi'\colon N\setminus Y \to L'\setminus Z$ of the 
same fine shape class as $\psi$ and taking the fine shape class of 
$\psi'\circ\phi$ --- the two compositions are of the same fine shape 
class from $X$ to $Z$.
\end{corollary}

Thus fine shape can be defined from $X$ to $Y$ without relying on any 
specific choice of spaces containing them; this is what differentiates 
fine shape from approaching maps.

By calling onto Lemma~\ref{extend_map_to_fSh} again, we easily see that 
fine shape is weaker than homotopy:

\begin{proposition}\label{extend_homotopy_to_fSh}
The fine shape class $[f]_{fSh}$ depends only on the homotopy class $[f]$. 
Composition of maps, or of homotopy classes, induces composition of the 
corresponding fine shape classes.
\end{proposition}

With all the pieces in place, we introduce

\begin{definition}\label{def_fSh_M}
The {\it fine shape category} $\mathrm{fSh(M)}$ is defined as follows:

(1)Its objects are the objects of $\mathrm{M}$ (metrizable topological spaces);

(2)Morphisms from $X$ to $Y$ are the fine shape classes from $X$ to $Y$;

(3)Composition is given by Corollary~\ref{fSh_composable};

(4)The identity morphism for a space $X$ is given by $[id_{X}]_{fSh}$, 
the fine shape class of $id_{X}$.
\end{definition}

Now Proposition~\ref{extend_homotopy_to_fSh} has the following

\begin{corollary}
There is a functor from $\mathrm{hM}$ to $\mathrm{fSh(M)}$ that is 
constant on objects and sends every homotopy class to the sole fine 
shape class induced by it.
\end{corollary}

\begin{notation}
For an approaching map $\phi\colon M\setminus X \to N\setminus Y$, we 
shall use $[\phi]$ to denote the fine shape class from $X$ to $Y$ 
defined by $\phi$. For an actual map $f\colon X \to Y$, we shall use 
$[f]_{fSh}$ to denote the fine shape class (again from $X$ to $Y$) 
defined by $f$, or by its homotopy class $[f]$. The set of all fine 
shape classes from $X$ to $Y$, we denote by $[X,Y]_{fSh}$.
\end{notation}

We finish this subsection by discussing maps of $\mathrm{M}$ that induce 
isomorphisms in $\mathrm{fSh(M)}$:

\begin{definition}
A map $f\colon X \to Y$ is called a {\it fine shape equivalence} 
whenever $[f]_{fSh}$ is an isomorphism.
\end{definition}

In regards to our search for a model structure on $\mathrm{M}$ for which 
$\mathrm{fSh(M)}$ is the homotopy category, fine shape equivalences must 
clearly be the weak equivalences. To that end we state the obvious

\begin{proposition}\label{fSh_eq_props}
(1)Fine shape equivalences satisfy the 2-out-of-3-property: for any maps 
$f\colon X \to Y$ and $g\colon Y \to Z$, if any two of the maps $f$, 
$g$, and $gf$ are fine shape equivalences, then so is the third;

(2)Fine shape equivalences are closed under retracts: for any 
commutative diagram
$$\xymatrix{
A \ar[r]_{i} \ar@/^/[rr]^{id_{A}} \ar[d]_{g} & X \ar[d]_{f} \ar[r]_{r} & A \ar[d]^{g} \\
B \ar[r]^{j} \ar@/_/[rr]_{id_{B}} & Y \ar[r]^{s} & B,
}$$
\noindent if $f$ is a fine shape equivalence, then so is $g$.
\end{proposition}

\subsection{Left fraction localization}
\label{def_left_fractions}

In general, a category can be localized at a selected class of 
its morphisms by allowing to invert each morphism in this class. Under 
some conditions on the selected class, the 
localization may have a simple representation. In the following, we make 
use of a specific case of this, the left fraction localization, as used 
by Cathey~\cite[Theorem (1.10)]{Cathey1981}, which 
references~\cite{Gabriel1967} and~\cite{Schubert1972}:

\begin{definition}\label{def_Ore}
Assume $\mathrm{C}$ is a category, and $\Sigma$ is a class of morphisms 
of $\mathrm{C}$. We say that $\Sigma$ {\it satisfies the left 
invertibility conditions} (sometimes called the left Ore conditions) if:

(1)$\Sigma$ contains all isomorphisms and is closed under morphism 
composition;

(2)For any morphisms $i\colon A \to X$ and $u\colon A \to Y$ with 
$i \in \Sigma$, there exists an object $Z$ along with morphisms 
$j\colon Y \to Z$ and $v\colon X \to Z$ such that $ju = vi$ and 
$j \in \Sigma$:

$$\xymatrix{
A \ar[r]^{u} \ar[d]|{\circ}_{i} & Y \ar@{.>}[d]|{\circ}^{j} \\
X \ar@{.>}[r]_{v} & Z
}$$

(3)For any morphisms $u,v\colon X \to Y$ such that $ui = vi$ for some 
$i\colon A \to X$ with $i \in \Sigma$, there exist an object $Z$ and 
a morphism $j\colon Y \to Z$ such that $ju = jv$ and $j \in \Sigma$:

$$\xymatrix{
A \ar[r]|{\circ}^{i} & X \ar@/^/[r]^{u} \ar@/_/[r]_{v} & Y \ar@{.>}[r]|{\circ}^{j} & Z
}$$
\end{definition}

Here and below in the commutative diagrams we adopt the following

\begin{convention}
All diagrams we include are to be either assumed or proved commutative. 
In any diagram used to illustrate a property, solid arrows are assumed 
to exist, dotted arrows must be shown to exist, and arrows with circles 
on them belong to the select class of morphisms satisfying the left 
invertibility conditions; an object is assumed to exist if it has at 
least one solid arrow entering or exiting it, and must be shown to exist 
otherwise.
\end{convention}

\begin{definition}\label{def_left_cat}
Given a category $\mathrm{C}$ and a class of its morphisms $\Sigma$ 
satisfying the left invertibility conditions, the 
{\it category of left fractions} $\Sigma \backslash \mathrm{C}$ is 
defined as follows:

(1)Its objects are the objects of $\mathrm{C}$; 

(2)A morphism from $X$ to $Y$ is an equivalence class of cospans 
$X \xrightarrow{u} P \xleftarrow{i} Y$ (also denoted by 
$i\backslash u$), where $u$ (the ``numerator'') and $i$ 
(the ``denominator'') are morphisms of $\mathrm{C}$ with $i \in \Sigma$; 

(3)Two cospans $X \xrightarrow{u} P \xleftarrow{i} Y$ and 
$X \xrightarrow{v} Q \xleftarrow{j} Y$ belong to the same equivalence 
class precisely when there is a cospan 
$P \xrightarrow{f} R \xleftarrow{g} Q$ of morphisms of $\mathrm{C}$ such 
that $fu = gv$, $fi = gj$, and the latter morphism is in $\Sigma$:

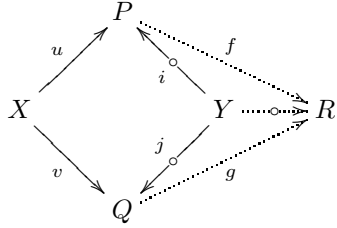
$$\xymatrix{
 & P \ar@{.>}[drr]^{f} & & \\
X \ar[ur]^{u} \ar[dr]_{v} & & Y \ar[ul]|{\circ}^{i} \ar[dl]|{\circ}_{j} \ar@{.>}[r]|{\circ} & R \\
 & Q \ar@{.>}[urr]_{g} & & \\
}$$

\noindent (in other words, 
$i\backslash u = fi\backslash fu = gj\backslash gv = j\backslash v$ 
whenever $fu = gv$ and $fi = gj$; note that neither $f$ nor $g$ has to 
be in $\Sigma$ here);

(4)The composition of two classes represented by cospans 
$X \xrightarrow{u} P \xleftarrow{i} Y$ and 
$Y \xrightarrow{v} Q \xleftarrow{j} Z$ can be obtained by choosing 
any cospan $P \xrightarrow{w} R \xleftarrow{k} Q$ with $wi = kv$ 
and $k \in \Sigma$ (which exists by the left invertibility 
condition~(2)), and taking the class represented by the cospan 
$X \xrightarrow{wu} R \xleftarrow{kj} Z$:

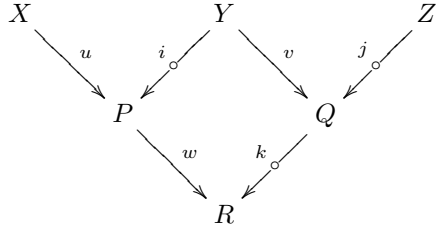
$$\xymatrix{
X \ar[dr]^{u} & & Y \ar[dl]|{\circ}_{i} \ar[dr]^{v} & & Z \ar[dl]|{\circ}_{j} \\
 & P \ar[dr]^{w} & & Q \ar[dl]|{\circ}_{k} & \\
 & & R & &
}$$
\noindent (in other words, 
$(j\backslash v)\circ(i\backslash u) = (kj\backslash kv)\circ(wi\backslash wu) = kj\backslash wu$; 
any suitable $k$ and $w$ will give the same composition class by 
condition~(3), and $kj \in \Sigma$ by condition~(1));

(5)The idenity isomorphism for an object $X$ is the class of the 
fraction $X \xrightarrow{id_{X}} X \xleftarrow{id_{X}} X$ (where 
$id_{X}$ is the identity in $\mathrm{C}$, and $id_{X} \in \Sigma$ by 
condition~(1)).
\end{definition}

\begin{convention}
Here and later we write a specific fraction to mean its equivalence 
class (similar to how fractions are used in arithmetic to represent 
rational numbers).
\end{convention}

We also state the universal property of the left fraction category:

\begin{proposition}\label{left_universal}
The left fraction category $\Sigma\backslash\mathrm{C}$, whenever it 
exists, is uniquely characterized by the following: 

(1)There is a functor 
$P\colon \mathrm{C} \to \Sigma\backslash\mathrm{C}$ 
acting on a morphism $f\colon X \to Y$ by $P(f) = id_{Y}\backslash f$; 

(2)$P$ sends $\Sigma$ into the class of isomorphisms: if 
$i\colon X \to Y$ is in $\Sigma$, then $id_{Y}\backslash i$ has an 
inverse $i\backslash id_{X}$;

(3)For any category $\mathrm{D}$ and any functor 
$F\colon \mathrm{C} \to \mathrm{D}$ sending $\Sigma$ into the class of 
isomorphisms in $\mathrm{D}$, there is a unique functor 
$Q\colon \Sigma\backslash\mathrm{C} \to \mathrm{D}$ (acting by 
$Q(i\backslash u) = (Fi)^{-1}\circ(Fu)$) such that $F = QP$. 

In short, $P$ is the universal functor out of $\mathrm{C}$ 
that sends all morphisms of $\Sigma$ into invertible morphisms.
\end{proposition}

\section{FDR-embeddings}
\label{sec_fdr}

As Cathey~\cite{Cathey1981} starts by defining what he called SSDR-maps, 
so do we follow by defining the fine shape version of those.

\begin{definition}\label{define_fdr}
Let $A$ be a closed subset of a space $X$, and assume that:

\begin{itemize}
  \item there exists an AR $M$ containing $X$ as a closed homotopy 
  negligible subset;
  \item there exists a closed subset $L$ of $M$ such that $L$ is an AR, 
  $L \cap X = A$, and $A$ is homotopy negligible in $L$;
  \item there exists an $(X \times [0,1])-X$-approaching map 
  $\Phi \colon (M \setminus X) \times [0,1] \to M \setminus X$ such that 
  $\Phi_{0} = id_{M \setminus X}$, 
  $\Phi_{1}(M \setminus X) = L \setminus A$, and 
  $\Phi_{t}|_{L \setminus A} = id_{L \setminus A}$ for all 
  $t \in [0,1]$.
\end{itemize}

In particular, considering $\Phi_{1}$ as an approaching map from 
$M\setminus X$ to $L\setminus A$, the fine shape class $[\Phi_{1}]$ 
(from $X$ to $A$) is the inverse to the fine shape class given by the 
embedding $A \subseteq X$ (which readily extends to the embedding 
$L \subseteq M$). In this case we shall say that the fine shape class 
$[\Phi]$ is a {\it fine shape strong deformation retraction} of $X$ on 
$A$. We shall also say that $A$ is a 
{\it fine shape strong deformation retract} of $X$, the inclusion 
$A \subseteq X$ is an {\it FDR-embedding} (which is simply a shorthand 
for ``embedding of a fine shape strong deformation retract''), 
and $\Phi$ is an {\it approaching strong deformation retraction} 
representing the fine shape strong deformation retraction $[\Phi]$.
\end{definition}

\begin{example}\label{fdr_examples}
An inclusion $X \times \{0\} \subset X \times [0,1]$ is clearly an 
FDR-embedding for any $X$.  Lemma~\ref{T_fdr} below provides a 
larger class of examples.
\end{example}

The rest of this section is concerned with the properties of 
FDR-embeddings and other statements that we will need to use them. 
Some of these simply extend results for SSDR-maps given by Cathey.

\begin{proposition}\label{H_extend}
Let $M$ be an AR, $L$ and $X$ closed subsets of $M$ with $L \cap X = A$. 
Assume there are homotopies $H \colon M \times [0,1] \to M$ and 
$G \colon L \times [0,1] \to L$ such that $H_{0} = id_{M}$, 
$G_{0} = id_{L}$, $H(M \times (0,1]) \subseteq M\setminus X$, and 
$G(L \times (0,1]) \subseteq L\setminus A$. Then there is a homotopy 
$F \colon M \times [0,1] \to M$ with $F_{0} = id_{M}$, 
$F(M \times (0,1]) \subseteq M\setminus X$, and 
$F|_{L \times [0,1]} = G$ (so in particular 
$F(L \times (0,1]) \subseteq L\setminus A$).
\end{proposition}

\begin{proof}
First define $F$ on $M \times \{0\} \cup L \times [0,1]$, a 
closed subset of $M \times [0,1]$, by $F|_{M \times \{0\}} = id_{M}$, 
$F|_{L \times [0,1]} = G$. As $M$ is an AR, this extends to some 
$F' \colon M \times [0,1] \to M$. Now to ensure that 
$F(M \times (0,1]) \subseteq M\setminus X$, take a continuous function 
$h \colon M \times [0,1] \to [0,1]$ with 
$h^{-1}(0) = M \times \{0\} \cup L \times [0,1]$, and define 
$F_{t}(m) := H_{h(m,t)} \circ F'_{t}(m)$.
\end{proof}

\begin{corollary}\label{single_H}
Under the conditions of Definition~\ref{define_fdr}, there exists a 
single homotopy $H \colon M \times [0,1] \to M$ with $H_{0} = id_{M}$, 
$H(M \times (0,1]) \subseteq M\setminus X$, and 
$H(L \times [0,1]) \subseteq L$. In other words, $(X,A)$ is homotopy 
negligible in $(M,L)$.
\end{corollary}

The next two lemmas are needed to extend homotopy negligibility property 
to mapping cylinders. These results are already practically contained 
in~\cite[Proposition~19.8]{Melikhov2022T}; we simply confirm the 
specific formulations we need.

\begin{lemma}\label{negl_mc}
Assume $Y$ is a closed homotopy negligible subset of a space $N$, $X$ 
and $L$ are closed subsets of a space $M$, $A = L \cap X$. Assume 
further that there is a homotopy $H\colon M \times [0,1] \to M$ such 
that $H_{0} = id_{M}$, $H(M \times (0,1]) \subseteq M\setminus X$, and 
$H(L \times [0,1]) \subseteq L$. Then for any map $f\colon L \to N$ 
such that $f^{-1}(Y) = A$, $MC_{X}(f|_{A})$ is homotopy negligible in 
$MC_{M}(f)$. In particular, with $L = M$, $A = X$, and 
$f\colon M \to N$, $MC(f|_{X})$ is homotopy negligible in $MC(f)$.
\end{lemma}

\begin{proof}
In addition to $H$, choose some homotopy 
$G\colon N \times [0,1] \to N$ with $G_{0} = id_{N}$ and 
$G(N \times (0,1]) \subseteq N\setminus Y$. Moreover, let $d_{N}$ be 
some metric on $N$ that provides its topology and is bounded by unity.

First, we will want to ``slow down'' $H$ if necessary to obtain a new 
homotopy $H'\colon M \times [0,1] \to M$ such that the three assumed 
properties of $H$ still apply, but in addition 
$d_{N}(f\circ H_{t}'(m),f(m)) \leq t$ for all $m \in M$ and 
$t \in [0,1]$. To achieve that, in $M \times [0,1] \times [0,1]$ 
consider the closed subset 
$C := \{(m,t,p) \mid d_{N}(f\circ H_{p}(m),f(m)) \geq t\}$ containing 
$M \times \{0\} \times \{0\}$ and define a function 
$\delta(m,t) := \frac{1}{2}d((m,t,0),C)$, taking any metric $d$ giving 
the product topology and bounded by unity. Now take 
$H_{t}'(m) := H_{\delta(m,t)}(m)$. Then $H_{0}' = id_{M}$ (as 
$(m,0,0) \in C$ for all $m \in M$), and for $(x,t) \in X \times (0,1]$, 
$H_{t}'(x) \in M\setminus X$ as $(x,t,0)$ is not in $C$. But also for 
all $(m,t) \in M\ \times [0,1]$, we have $(m,t,\delta(m,t)) \in C$ if 
and only if $\delta(m,t) = 0$ and so $f\circ H_{t}'(m) = f(m)$. Thus 
$d_{N}(f\circ H_{t}'(m),f(m)) \leq t$ for all $m$ and all $t$, as 
needed.

With this we shall define a homotopy 
$F\colon MC_{M}(f) \times [0,1] \to MC_{M}(f)$. First, $F$ restricts to 
$H'$ and $G$ on the standard embeddings of $M$ and $N$ respectively. 
Second, for a point $(l,s) \in L \times (0,1)$, we define 
$$
F_{t}(l,s) := 
\begin{cases}
  G_{t - 2s}\circ f \circ H_{2s}(l), & s \leq \frac{t}{2} \\
  (H_{t}'(m),\frac{2s - t}{2 - t}), & s > \frac{t}{2}
\end{cases}.
$$
Then $F_{0} = id_{MC_{M}(f)}$ and the image of $F_{t}$ does not 
intersect $MC_{X}(f|_{A})$ for $t > 0$. Moreover, $F$ is continuous: in 
the only inobvious case is if a sequence 
$(l_{k},s_{k})$ in $L \times (0,1)$ converges to $n \in N$, we have 
$d_{N}(f(l_{k}),n) \to 0$ and $s_{k} \to 0$. Now assume also 
$t_{k} \to t$. For those $k$ which satifsy $2s_{k} \leq t_{k}$, we have 
$d_{N}(f\circ H_{2s_{k}}'(l_{k}),f(l_{k})) \to 0$ and therefore 
$f\circ H_{2s_{k}}'(l_{k}) \to n$, so in addition to $s_{k} \to 0$ we 
have $G_{t_{k} - 2s_{k}}\circ f\circ H_{2s_{k}}'(l_{k}) \to G_{t}(n)$, 
which means that $F_{t_{k}}(l_{k},s_{k}) \to F_{t}(n)$. And those $k$ 
for which $2s_{k} > t_{k}$ are finite in number unless $t_{k} \to 0$, 
and in the latter case we obtain 
$d_{N}(f\circ H_{t_{k}}'(l_{k}),f(l_{k})) \to 0$, so also 
$f\circ H_{t_{k}}(l_{k}) \to n = G_{t}(n)$, so again 
$F_{t_{k}}(l_{k},s_{k}) \to G_{t}(n)$, as needed.
\end{proof}

\begin{lemma}\label{rel_cyl_h_negl}
Let $f \colon M \to N$ be a map, $X$ and $Y$ closed subsets 
of $M$ and $N$ respectively such that $f^{-1}(Y) = X$, $f|_{X}$ is 
perfect, and $Y$ is homotopy negligible in $N$. Then $Y$ is also 
homotopy negligible in $MC(f|X)$.
\end{lemma}

\begin{proof}
We have a homotopy $H \colon N \times [0,1] \to N$ with 
$H_{0} = id_{N}$ and $H(N \times (0,1]) \subseteq N\setminus Y$. Now 
on the usual cylinder $MC(f)$ define a homotopy 
$G \colon MC(f) \times [0,1] \to MC(f)$ by $G_{t}|_{N}(n) := H_{t}(n)$,

$$
G_{t}|_{M \times (0,1]}(m,s) := 
\begin{cases}
  (m,s-t), & t < s \\
  H_{t-s}\circ f(m), & t \geq s
\end{cases}
$$
 
Then in particular $G(MC(f|X) \times [0,1]) \subseteq MC(f|X)$. 
Moreover, $G(MC(f|X) \times (0,1]) \cap Y = \varnothing$, thus the 
restriction of $G$ gives a homotopy of $MC(f|X)$ proving $Y$ homotopy 
negligible in it.
\end{proof}

\begin{lemma}\label{fdr_AR_change}
Assume the inclusion $A \subseteq X$ is an FDR-embedding, with $M$, 
$L$, and $\Phi$ as in Definition~\ref{define_fdr}. For any other 
AR $L'$ containing $A$ as a closed homotopy negligible subset, we can 
always construct an AR $M'$ 
containing $X$ and $L'$ as closed subsets with $L' \cap X = A$ and $X$ 
homotopy negligible in $M'$, along with an approaching strong 
deformation retraction $\Phi'$ of $M'\setminus X$ onto $L'\setminus A$. 
\end{lemma}

\begin{corollary}\label{fdr_AR_property}
Assume that every space $A$ can be embedded as a closed homotopy 
negligible subset in an AR $L$ such that some property $P$ holds for 
either $L$, or the embedding $A \subseteq L$. Then for every 
FDR-embedding $A \subseteq X$, we can have $M$, $L$, and $\Phi$ as in 
Definition~\ref{define_fdr} such that the property $P$ holds for $L$ or 
$A \subseteq L$ correspondingly.
\end{corollary}

\begin{remark}
Note that the same cannot, in general, be said of $M$ or $X \subseteq M$ 
in this situation.
\end{remark}

\begin{proof}[Proof (of Lemma~\ref{fdr_AR_change})]
The identity map $id_{A}$ extends to some map 
$\bar{id}_{A} \colon L \to L'$ with $\bar{id}_{A}^{-1}(A) = A$. Now 
$MC(\bar{id}_{A}|A)$ is an AR, and so is 
$M' := M \cup_{L} MC(\bar{id}_{A}|A)$. By applying $\Phi$ to 
$M\setminus X$ while keeping every point of 
$MC(\bar{id}_{A}|A)\setminus A$ constant (these agree on 
$L\setminus A$), we obtain an approaching strong deformation retraction 
of $M'\setminus X$ onto $MC(\bar{id}_{A}|A)\setminus A$, which then 
approaching strong deformation retracts onto $L'\setminus A$. 

To prove $X$ homotopy negligible in $M'$, first use 
Corollary~\ref{single_H} to construct a homotopy 
$F \colon M \times [0,1] \to M$ with 
$F_{0} = id_{M}$, $F(M \times (0,1]) \subseteq M\setminus X$, and 
$F(L \times [0,1]) \subseteq L$; further use Lemma~\ref{rel_cyl_h_negl} 
to construct a homotopy 
$G \colon MC(\bar{id}_{A}|A) \times [0,1] \to MC(\bar{id}_{A}|A)$ with 
$G_{0} = id_{MC(\bar{id}_{A}|A)}$ and 
$G(MC(\bar{id}_{A}|A) \times (0,1]) \subseteq MC(\bar{id}_{A}|A)\setminus A$. 
Then by Proposition~\ref{H_extend} there is a homotopy 
$F' \colon MC(\bar{id}_{A}|A) \times [0,1] \to MC(\bar{id}_{A}|A)$ with 
$F'_{0} = id_{MC(\bar{id}_{A}|A)}$, 
$F'(MC(\bar{id}_{A}|A) \times (0,1]) \subseteq MC(\bar{id}_{A}|A)\setminus A$, 
and $F'|_{L \times [0,1]} = F|_{L \times [0,1]}$. Now combine $F$ and 
$F'$ into a single homotopy $F'' \colon M' \times [0,1] \to M'$.
\end{proof}

\begin{proposition}\label{fdr_closed}
Every identity map is an FDR-embedding. If $i\colon X \to Y$ and 
$j\colon Y \to Z$ are FDR-embeddings, so is $ji$.
\end{proposition}

\begin{proof}
Identity part is clear.

For composition, first $ji$ is a closed embedding. Now assume $i$ being 
an FDR-embedding is proven (as per Definition~\ref{define_fdr}) with 
ARs $N \supset Y$, $M \supset X$, and an approaching strong deformation 
retraction $\Phi\colon (N\setminus Y)\times [0,1] \to N\setminus Y$, and 
$j$ with ARs $L' \supset Z$, $N' \supset Y$, and 
$\Psi'\colon (L'\setminus Z)\times [0,1] \to L'\setminus Z$; all 
inclusions are those of closed homotopy negligible subsets. By the 
previous lemma, construct an AR $L''$ containing $N$ and $Z$ as closed 
subsets with $N \cap Z = Y$, $Z$ a Z-set in $L''$, and $L''\setminus Z$ 
approaching strong deformation retracting onto $N\setminus Y$ by some 
$\Psi''$. Then the composition $\Phi \circ \Psi''$ is well-defined (as 
the image of $\Psi''_{1}$ is $N\setminus Y$), and is an approaching 
strong deformation retraction of $L''\setminus Z$ onto $M\setminus X$.
\end{proof}

\begin{lemma}
Let $A \subseteq X$ be a closed embedding with $A$ nonempty. Then there 
exist an AR $M$ containing $X$ as a closed subset, a closed subset 
$L \subseteq M$ such that $L \cap X = A$ and $L$ is an AR (in fact, $L$ 
can be chosen to be a strong deformation retract of $M$), and a homotopy 
$H\colon M \times [0,1] \to M$ such that $H_{0} = id_{M}$, 
$H(M \times (0,1]) \subseteq M\setminus X$, and 
$H(L \times [0,1]) \subseteq L$.
\end{lemma}

\begin{proof}
Take $M'$ to be any AR containing $X$ as a closed subset, and take any 
continuous function $h\colon M' \to [0,1]$ such that $h^{-1}(0) = A$.
Now let $M := M' \times [0,1]$ with $X = X \times \{0\} \subset M$, 
$L := \{(m,s) \in M \mid s \geq h(m)\}$, $d_{M'}$, and 
$H_{t}(m,s) := (m,\min\{s+t,1\})$
\end{proof}

\begin{lemma}\label{T_fdr}
Let $A$ be a closed subset of a space $X$. Then the inclusion 
of $A \times [0,1] \cup X \times \{0\}$ in $X \times [0,1]$ is an 
FDR-embedding.~(cf.~\cite[Corollary~1.6(b)]{Melikhov2022F},~\cite[Corollary~(1.6)]{Cathey1981})
\end{lemma}

\begin{remark}\label{fdr_nontrivial}
This provides nontrivial examples of FDR-embeddings. After all, a closed 
embedding $A \subseteq X$ is a cofibration if and only if 
$A \times [0,1] \cup X \times \{0\}$ is a retract of $X \times [0,1]$, 
and there exist closed embeddings of metrizable spaces that are not 
cofibrations. Note also that the fine shape classes of the approaching 
retractions are therefore not induced, in general, by any ordinary maps.
\end{remark}

\begin{proof}[Proof (of Lemma~\ref{T_fdr})]
If $A$ is empty, then $X \times [0,1]$ has a strong 
deformation retraction onto $X \times \{0\}$, which readily extends to 
a strong deformation retraction of $M \times [0,1]$ onto 
$M \times \{0\}$ for any AR $M$ containing $X$ as a closed homotopy 
negligible subset. 

Now assume $A$ nonempty. Take $M$, $L$, and $H$ given by the previous 
lemma. Let $N$ to be the relative metrizable mapping cylinder of the 
embedding $L \subseteq M$ by $A$. Then $N$ is an AR, and there is a 
homotopy $G \colon N \times [0,1] \to N$ with $G_{0} = id_{N}$, 
$G(N \times (0,1]) \subseteq N\setminus X$, and 
$G(L \times [0,1]) \subseteq L$: first apply Lemma~\ref{rel_cyl_h_negl} 
using $H$ to prove that $X$ is closed and homotopy negligible in $N$, 
then apply Proposition~\ref{H_extend} with $H|_{L \times [0,1]}$ as the 
second homotopy to obtain $G$ with 
$G|_{L \times [0,1]} = H|_{L \times [0,1]}$. Now $N \times [0,1]$ is an 
AR containing $X \times [0,1]$ as a closed homotopy negligible subset, 
and $L \times [0,1] \cup N \times \{0\}$ an AR in it intersecting 
$X \times [0,1]$ precisely by $A \times [0,1] \cup X \times \{0\}$, 
which is closed and homotopy negligible in it too (this uses the fact 
that $G(L \times [0,1]) \subseteq L$). 

For the approaching strong deformation retraction, note that 
$N\setminus X$ is homeomorphic to the (nonrelative) metrizable mapping 
cylinder of the embedding of $L\setminus A$ into $M\setminus X$. Then, 
as per Remark~\ref{MC_remarks}(3), we have a strong deformation 
retraction of $(N\setminus X) \times [0,1]$ onto 
$(L\setminus A) \times [0,1] \cup (N\setminus X) \times \{0\}$; this 
retraction is correctly approaching, thus proving the lemma.
\end{proof}

\begin{corollary}\label{II_fdr}
If $A \subseteq X$ is an FDR-embedding, then so is 
$A \times [0,1] \cup X \times \{0,1\} \subseteq X \times [0,1]$~(cf.~\cite[Corollary~(1.6)]{Cathey1981}).
\end{corollary}

\begin{proof}
Take spaces $M \supset X$ and $L \supset A$ and map $\Phi$ 
as in the definition of an FDR-embedding, and take 
$L \times [0,1] \cup M \times \{0,1\} \subseteq M \times [0,1]$ to be 
the ARs, clearly containing $A \times [0,1] \cup X \times \{0,1\}$ and 
$X \times [0,1]$ respectively as closed homotopy negligible subsets. Now 
define 
$\tilde{\Phi} \colon (M \times [0,1] \setminus X \times [0,1]) \times [0,1] \to M \times [0,1] \setminus X \times [0,1]$ 
as a composition of $\tilde{\Phi}^{(1)}$ followed by 
$\tilde{\Phi}^{(2)}$. Here, 
$\tilde{\Phi}^{(1)}_{t}(m,s) := (\Phi_{(1 - |2s - 1|)t}(m),s)$, 
and then $\tilde{\Phi}^{(2)}$ applies the previous lemma to 
$(M \setminus X) \times [0,\frac{1}{2}]$ and 
$(M \setminus X) \times [\frac{1}{2},1]$ separately. The two maps for 
halves can be glued together into a single $\tilde{\Phi}^{(2)}$ because 
the image of $\tilde{\Phi}^{(1)}_{1}$ does not contain any points of the 
intersection $(M \setminus X) \times \{\frac{1}{2}\}$ that do not belong 
to $(L \setminus A) \times \{\frac{1}{2}\}$, and those stay unchanged 
under both maps, which therefore coincide on the intersection. Then 
$\tilde{\Phi}$ is an approaching strong deformation retraction, as 
needed.
\end{proof}

The final part of this section is not used further in the present work, 
yet it pertains directly to the search for the model structure, as 
explained in the introduction:

\begin{theorem}\label{WC_is_fdr}
Let $i\colon A \to X$ be a closed embedding. Then $i$ is an 
FDR-embedding if and only if $i$ is a fine shape equivalence.
\end{theorem}

\begin{corollary}\label{WC_is_fdr_cor}
(1)A closed embedding homotopic to an FDR-embedding is itself an 
FDR-embedding;

(2)Let $i\colon X \to Y$ and $j\colon Y \to Z$ be closed embeddings. If 
any two of $i$, $j$, and $ji$ are FDR-embeddings, then so is the 
third;

(3)Retracts of FDR-embeddings are FDR-embeddings.
\end{corollary}

\begin{remark}
(1)The proof of the theorem uses the same idea as for the property 
of cofibrations briefly remarked upon in~\cite{Strom1968} (the first 
paragraph of section~3.), or proven explicitly 
in~\cite[Lemma~A.2]{Melikhov2022T}.

(2)Cathey proves the properties of the theorem and the corollary 
for SSDR-maps between compacta (where fine shape reduces to strong 
shape), albeit in different order and with different 
methods~\cite[(1.15), (1.7), (1.5)]{Cathey1981}. 

(3)Cathey already mentions (in the introduction) that embeddings 
seem to be the cofibrations for compact strong shape; the same can be 
said for closed embeddings and fine shape over all metrizable spaces 
(Cathey did not need to specify ``closed'' because continuous maps 
between compacta are always closed). Thus the theorem states that 
FDR-embeddings are the acyclic, or trivial, cofibrations.
\end{remark}

\begin{proof}[Proof (of Theorem~\ref{WC_is_fdr})]
The ``only if'' part is automatic. We deal with the ``if'' part.

First take some ARs $M$ and $L$ containing $X$ and $A$ respectively as 
closed homotopy negligible subsets, as well as some map 
$\bar{\i}\colon L \to M$ extending $i$ and such that 
$\bar{\i}^{-1}(X) = A$. We can assume that $\bar{\i}$ is a closed 
embedding, because we can replace $M$ with the relative metrizable 
cylinder $MC(\bar{\i}|A)$; for the same reason we can assume that there 
is an approaching strong deformation retraction of 
$(M\setminus X) \times [0,1]$ onto 
$(L\setminus A) \times [0,1] \cup (M\setminus X) \times \{0\}$ as in the 
proof of Lemma~\ref{T_fdr}. 

Since $i$ is a fine shape equivalence, the inverse of the fine shape 
class $[i]_{fSh} = [\bar{\i}|_{L\setminus A}]$ exists and can be 
represented by some approaching map 
$\phi\colon M\setminus X \to L\setminus A$, along with approaching 
homotopies $\Phi\colon (M\setminus X) \times [0,1] \to M\setminus X$ and 
$\Psi\colon (L\setminus A) \times [0,1] \to L\setminus A$ such that 
$\Phi_{0} = id_{M\setminus X}$, $\Phi_{1} = \bar{\i}\phi$, 
$\Psi_{0} = id_{L\setminus A}$, and 
$\Psi_{1} = \phi\bar{\i}|_{L\setminus A}$. In particular, the image of 
$\Phi_{1}$ lies in $L\setminus A$; so define an approaching map 
$\omega\colon (L\setminus A) \times [0,1] \cup (M\setminus X) \times \{0\} \to L\setminus A$ 
by $\omega|_{M\setminus X} = \Phi_{1}$ and 
$\omega|_{(L\setminus A) \times [0,1]}(l,t) = \Psi_{1-t}(l)$, so that 
$\omega(l,1) = l$ for all $l \in L\setminus A$; this is possible 
because, again, for all $l \in L$, we have $\Phi_{1}(l) = \phi(l)$, as 
$\bar{i}$ is a closed embedding. 

By the approaching strong deformation retraction mentioned above, we can 
extend $\omega$ to an approaching map 
$\Omega\colon (M\setminus X) \times [0,1] \to L\setminus A$; in 
particular, $\Omega_{1}$ is an approaching retraction of $M\setminus X$ 
onto $L\setminus A$. Compose $\Phi_{t}$ and $\Omega_{t}$ into a 
map $\Omega'\colon (M\setminus X) \times [0,1] \to M\setminus X$ with 

$$\Omega_{t}' = 
\begin{cases}
\Phi_{2t},\, t < \frac{1}{2} \\
\Omega_{2t-1},\, t \geq \frac{1}{2}
\end{cases}$$.

Thus $\Omega'$ is an approaching deformation retraction joining 
$\Omega_{0}' = \Phi_{0} = id_{M\setminus X}$ and 
$\Omega_{1}' = \Omega_{1}$.

Now from the proof of Lemma~\ref{T_fdr} again, there is a map
$\theta\colon (M\setminus X) \times [0,1] \times [0,1] \to 
(M\setminus X) \times [0,1]$ which is an approaching strong deformation 
retraction onto 
$(L\setminus A) \times [0,1] \cup (M\setminus X) \times \{1\}$.
Define $\Theta\colon (M\setminus X) \times [0,1] \to M\setminus X$ by 
$\Theta(m,t) := \Omega'\circ\theta(m,0,t)$. Then 
$\Theta_{0} = \Omega'_{0} = id_{M\setminus X}$, the image of 
$\Theta_{1}$ is in $L\setminus A$, and for all $l \in L\setminus A$, 
$\Theta(l,t) = \Omega'\circ\theta_{t}(l,0) = \Omega'(l,0) = l$. Thus 
$\Theta$ is an approaching strong deformation retraction of 
$M\setminus X$ onto $L\setminus A$, proving $i$ to be an 
FSSDR-embedding.
\end{proof}

\section{The left fraction localization}
\label{sec_fdr_localization}

\begin{proposition}\label{fdr_Ore_2}
Let $i\colon A \to X$ be an FDR-embedding and $f \colon A \to Y$ be a 
map. Then there exists a space $Z$ with maps 
$j\colon Y \to Z$ and $g\colon X \to Z$ such that $gi \simeq jf$ and 
$j$ is an FDR-embedding.
\end{proposition}

\begin{proof}
Take $Z$ to be $MC_{X}(f)$. Take $j$ and $f$ to be 
the usual inclusions of $Y$ and $X$ in $MC_{X}(f)$. Clearly 
$gi \simeq jf$ as usual for a mapping cylinder.

To show $j$ is an FDR-embedding, first for $A \subseteq X$ we have some 
$M$, $L$, and $\Phi$ as in Definition~\ref{define_fdr}. Also let $N$ be 
an AR containing $Y$ as a closed homotopy negligible subset. The map $f$ 
extends to a map $\bar{f} \colon L \to N$ with $\bar{f}^{-1}(Y) = X$. 
For an AR containing $Z$, take $MC_{M}(\bar{f})$, which contains 
$MC_{X}(f)$ as a closed homotopy negligible subset by 
Lemma~\ref{negl_mc}. Now by applying $\Phi$ to $M\setminus X$ while 
keeping each other point of $MC_{M}(\bar{f})$ in place, 
$MC_{M}(\bar{f})\setminus MC_{X}(f)$ approaching strong deformation 
retracts onto $MC(\bar{f})\setminus MC(f)$. The latter in turn 
approaching strong deformation retracts onto its base (this even extends 
to a strong deformation retraction of $MC(f)$ onto $N$, of course). 
Composing the two proves $j$ to be an FDR-embedding, as required.
\end{proof}

\begin{remark}
Clearly $Z$ here is the pushout of the diagram in $\mathrm{hM}$. Thus 
the proposition also says that in $\mathrm{hM}$, the pushout of a 
homotopy class represented by an FDR-embedding is itself such a class.

Note also that $\mathrm{M}$, unlike $\mathrm{hM}$, does not even have 
all pushouts (the topological pushout --- the adjunction space --- of 
metrizable spaces may not be metrizable). This is why we need to work 
with the homotopy category, same as Cathey and others before.
\end{remark}

\begin{proposition}\label{fdr_Ore_3}
Let $i \colon A \to X$ be an FDR-embedding and $u,v\colon X \to Y$ be 
two maps such that there is a homotopy $G\colon ui \simeq vi$. Then 
there exist a space $Z$ and an FDR-embedding $j\colon Y \to Z$ such 
that $ju \simeq jv$.
\end{proposition}

\begin{proof}
Maps $u$, $v$, and $G$ combine into a map 
$f \colon A \times [0,1] \cup X \times \{0,1\} \to Y$ by 
$f|_{X \times \{0\}} := u$, $f|_{X \times \{1\}} := v$, and 
$f|_{A \times [0,1]} := G$. The inclusion 
$A \times [0,1] \cup X \times \{0,1\} \subseteq X \times [0,1]$ is an 
FDR-embedding by Corollary~\ref{II_fdr}, so we can apply the previous 
proposition to this inclusion and $f$ to obtain a space $Z$, an 
FDR-embedding $j\colon Y \to Z$, and a homotopy 
$F\colon X \times [0,1] \to Z$ such that $F_{0} \simeq ju$ and 
$F_{1} \simeq jv$.
\end{proof}

Now we can define the left fraction category we shall use.

\begin{definition}
In $\mathrm{hM}$, we denote by $[FDR]$ the class of morphisms consisting 
of all homotopy classes that can be represented by FDR-embeddings 
(cf.~Corollary\ref{WC_is_fdr_cor}(1)). 
\end{definition}

One important property here is that if a homotopy class $[i]$ is in 
$[FDR]$, then the fine shape class $[i]_{fSh}$ is an isomorphism, and 
thus has an inverse $[i]_{fSh}^{-1}$.

\begin{theorem}
In $\mathrm{hM}$, the class $[FDR]$ satisfies the conditions for left 
invertibility of Definition~\ref{def_Ore}.
\end{theorem}

\begin{proof}
Condition~(1) is by Proposition~\ref{fdr_closed}. Conditions~(2) and~(3) 
are by Propositions~\ref{fdr_Ore_2} and~\ref{fdr_Ore_3}.
\end{proof}

\begin{corollary}\label{left_exists}
There exists a well-defined category of left fractions 
$[FDR]\backslash \mathrm{hM}$.
\end{corollary}

\section{The S functor}
\label{sec_functor_S}

The primary goal of this work is to prove $\mathrm{fSh(M)}$ and 
$[FDR]\backslash \mathrm{hM}$ to be isomorphic categories, thus 
extending the result of Cathey. So we need to define two functors 
between these that will be inverse to each other. One is readily 
defined.

\begin{theorem}\label{define_S}
There is a functor 
$S\colon [FDR]\backslash \mathrm{hM} \to \mathrm{fSh(M)}$, uniquely 
defined by the following properties:

(1)$S$ is constant on objects;

(2)Whenever a left fraction class can be represented by the fraction 
$[i]\backslash [u]$, $S([i]\backslash [u])$ is the fine shape class 
$[i]_{fSh}^{-1}\circ [u]_{fSh}$.
\end{theorem}

\begin{proof}
The action of $S$ is already defined by the properties, so 
we need only check that $S([i]\backslash [u])$ does not depend on the 
choice of fraction representing the given fraction class, and that $S$ 
respects identity and composition.

First, equivalent fractions are sent to the same fine shape 
class. To wit, assume $X \xrightarrow{u} Z \xleftarrow{i} Y$ and 
$X \xrightarrow{v} W \xleftarrow{j} Y$ are cospans of specific maps such 
that the fractions of their homotopy classes are in the same equivalence 
class. This means there exist $Z \xrightarrow{f} P \xleftarrow{g} W$ 
with $fu \simeq gv$ and $fi \simeq gj$, and the latter class 
belonging to $[FDR]$. Now since $i$, $j$, and $fi \simeq gj$ are all 
fine shape equivalencies, so must be $f$ and $g$ by 
Proposition~\ref{fSh_eq_props}(1) (albeit the homotopy classes $[f]$ and 
$[g]$ might not be in $[FDR]$). Then we have 
$[i]_{fSh}^{-1}\circ[u]_{fSh} = [fi]_{fSh}^{-1}\circ[fu]_{fSh} = 
[gj]_{fSh}^{-1}\circ[gv]_{fSh} = [j]_{fSh}^{-1}\circ[v]_{fSh}$, 
as required.

Second, an identity left fraction class clearly maps to the 
corresponding identity fine shape class.

Third, consider the composition. If we have fraction classes 
represented by specific maps $X \xrightarrow{u} P \xleftarrow{i} Y$ and 
$Y \xrightarrow{v} Q \xleftarrow{j} Z$ respectively, then the 
composition $([j]\backslash [v]) \circ ([i]\backslash [u])$ can be 
represented by $X \xrightarrow{fu} W \xleftarrow{kj} Z$ for some 
$P \xrightarrow{f} W \xleftarrow{k} Q$, where $fi \simeq kv$, and $k$ is 
an FDR-embedding. In that case we obtain 
$S(([j]\backslash [v]) \circ ([i]\backslash [u])) = [kj]_{fSh}^{-1}\circ[fu]_{fSh} = 
[j]_{fSh}^{-1}\circ[k]_{fSh}^{-1}\circ[fi]_{fSh}\circ[i]_{fSh}^{-1}\circ[u]_{fSh} = 
[j]_{fSh}^{-1}\circ[k]_{fSh}^{-1}\circ[kv]_{fSh}\circ[i]_{fSh}^{-1}\circ[u]_{fSh} = 
[j]_{fSh}^{-1}\circ[v]_{fSh}\circ[i]_{fSh}^{-1}\circ[u]_{fSh}$, which is 
exactly the composition 
$S([j]\backslash [v])\circ S([i]\backslash [u])$.
\end{proof}

Now our goal is to define the inverse functor to $S$. For this we shall 
need an additional construction.

\section{Approaching and fine shape cylinders}
\label{sec_AC_FC}

The basic idea is that given an approaching map 
$\phi\colon M\setminus X \to N\setminus Y$ with $X$ {\it additive} 
homotopy negligible in $M$ (see Definition~\ref{def_H_add}), we shall 
construct an object $(AC(\phi),FC(\phi))$ of $\mathrm{MAppr}$ that will 
be somewhat similar to a metrizable mapping cylinder in $\mathrm{M}$, or 
to a mapping cylinder in the topological category.

The primary property of $(AC(\phi),FC(\phi))$ is to be this: 
we want to have inclusions of pairs 
$\bar{u}\colon (M,X) \to (AC(\phi),FC(\phi))$ 
and $\bar{\i}\colon (N,Y) \to (AC(\phi),FC(\phi))$ that satisfy 
$\bar{u}^{-1}(AC(\phi)) = X$ and $\bar{\i}^{-1}(FC(\phi)) = Y$; and for 
any approaching homotopy 
$\Psi\colon (M \times [0,1])\setminus (X \times [0,1]) \to L\setminus Z$ 
between $\Psi_{0} = \psi\colon M\setminus X \to L\setminus Z$ and 
$\Psi_{1} = \chi\circ\phi$ for some 
$\chi\colon N\setminus Y \to L\setminus Z$, we want to have an 
approaching map 
$\omega\colon (AC(\phi),FC(\phi)) \to L\setminus Z$ such that 
$\omega\circ\psi = \bar{u}|_{M\setminus X}$ and
$\omega\circ\chi = \bar{\i}|_{N\setminus Y}$.

For compact spaces, it is possible to dispense with the requirement of 
additive homotopy, and define a metrizable topology on 
$M \times (0,1] \cup N$ making the pair 
$(M \times (0,1] \cup N,X \times (0,1] \cup Y)$ into an object having 
the desired property. Explicit definition is given in particular by 
Mrozik~\cite{Mrozik1990} (the pair $(D(\phi),C(\phi))$ there, defined in 
the beginning of section~2.), who refers to an earlier equivalent 
construction of Ferry~\cite{Ferry1980}. In fact, Cathey's 
work~\cite{Cathey1981} already effectively makes use of the same 
construction (note the space $Z$ in the proof of Corollary~(2.7) there), 
though no name or notation is given.

Unfortunately, it seems that this construction cannot be directly 
applied to noncompact spaces; see the appendix for a discussion of the 
problem. Instead, we shall obtain the desired pair of spaces by taking 
larger --- and more convoluted --- underlying sets.

The rest of this section, along with the next one, shall use the 
notation from the following

\begin{definition}\label{define_AC_FC}
Consider an approaching map $\phi \colon M\setminus X \to N\setminus Y$ 
with $X$ additive homotopy negligible in $M$, and choose some specific 
additive homotopy $H \colon M \times [0,1] \to M$ with $H_{0} = id_{M}$ 
and $H(M \times (0,1]) \subseteq M\setminus X$.

In the metrizable join $M \star N$, take the following four subsets:

$$
C_{+}(\phi, H) := \{(m,s,\phi\circ H_{s}(m)) \mid m \in M\setminus X, s \in (0,1)\}
$$
$$
C_{+,0}(\phi, H) := \{(x,s,\phi\circ H_{s}(x)) \mid x \in X, s \in (0,1)\}
$$
$$
C_{-}(\phi) := \{(m,s,n) \mid m \in M\setminus X, s \in (-1,0], n \in N\setminus Y\}
$$
$$
C_{-,0}(\phi) := \{(x,s,y) \mid x \in X, s \in (-1,0], y \in Y\}
$$

Of course, $C_{-}$ and $C_{-,0}$ depend only on $(M,X)$ and $(N,Y)$, 
and not on the action of $\phi$ or on $H$. In addition, $C_{-}(\phi)$ is 
the metrizable mapping cylinder of the projection

$$(M\setminus X) \times (N\setminus Y) \ni (m,n) \mapsto n \in (N\setminus Y),$$ 

whereas $C_{-,0}(\phi)$ is the metrizable mapping cylinder of the projection

$$X \times Y \ni (x,y) \mapsto y \in Y$$.

Finally, we define the {\it approaching cylinder} of $(\phi,H)$ 
to be the subset 
$AC(\phi,H) := M \cup C_{+} \cup C_{+,0} \cup C_{-} \cup C_{-,0} \cup N$ 
of $M \star N$, and the {\it fine shape cylinder} of $(\phi,H)$ 
to be $FC(\phi,H) := X \cup C_{+,0} \cup C_{-,0} \cup Y$.
\end{definition}

\begin{remark}\label{FC_blow_up}
The compact approaching cylinder $C(\phi)$ of Mrozik has 
$X \times (0,1] \cup Y$ as its underlying set (same as the usual mapping 
cylinder). Our construction of $FC(\phi,H)$, on the other hand, inserts 
the ``negative half'' $C_{-,0}(\phi) = X \times (-1,0] \times Y$. In 
essence, we take the cylinder of Mrozik, ``blow up'' $Y$ into 
$X \times Y$ (replacing every point of $Y$ by a whole copy of $X$), and 
then attach the mapping cylinder of the projection $X \times Y \to Y$ 
(as we wouldn't have an embedding of $Y$ otherwise).

%This is also why our definition of the metrizable join 
%(Definition~\ref{def_join}) uses $[-1,1]$ rather than $[0,1]$, though 
%clearly the ``negative half'' could be chosen to run only a small part 
%of the $s \in [-1,1]$ coordinate of the join.
\end{remark}

\begin{proposition}\label{AC_FC_props}
In the notation of previous definition, 

(1)$M \cup C_{+}(\phi,H) \cup C_{+,0}(\phi,H)$ is homeomorphic to 
$M \times (0,1]$, and $AC(\phi,H)\setminus FC(\phi,H)$ is homeomorphic 
to the union $MC(p) \cup MC(q)$, where 
$p\colon M\setminus X \to (M\setminus X) \times (N\setminus Y)$ and 
$q\colon (M\setminus X) \times (N\setminus Y) \to N\setminus Y$ are 
defined by $p(m) := (m,\phi(m))$ and $q(m,n) := n$, and the union is 
along the image of $p$ (equivalently, the graph of $\phi$) in 
$(M\setminus X) \times (N\setminus Y)$;

(2)$FC(\phi,H)$ is a closed homotopy negligible subset of 
$AC(\phi,H)$ --- additive homotopy negligible if $Y$ is additive 
homotopy negligible in $N$;

(3)The choice of $H$ does not affect the isomorphism class of 
$(AC(\phi,H),FC(\phi,H))$ in $\mathrm{MAppr}$;

(4)If $M$ and $N$ are ARs, then $AC(\phi,H)$ is also an AR, and the 
fine shape class of $FC(\phi,H)$ is also independent of the choice of $H$.
\end{proposition}

\begin{proof}
(1)First, we have the bijection 
$$
M \cup C_{+}(\phi,H) \cup C_{+,0}(\phi,H) \ni (m,s,\phi\circ H_{s}(m)) \leftrightarrow (m,s) \in M \times (0,1].
$$
\noindent As the spaces are metrizable, we can prove continuity in 
terms of sequence convergence, as per Remark~\ref{join_remarks}(5). Thus 
we prove that 
$(m_{k},s_{k},\phi\circ H_{s_{k}}(m_{k})) \to (m,s,\phi\circ H_{s}(m))$ 
in $M \cup C_{+}(\phi,H) \cup C_{-}(\phi,H)$ if and only if 
$(m_{k},s_{k}) \to (m,s)$ in $M \times (0,1]$. The ``only if'' part is 
automatic, and for the ``if'' part, $(m_{k},s_{k}) \to (m,s)$ implies 
$H_{s_{k}}(m_{k}) \to H_{s}(m)$, thus 
$\phi\circ H_{s_{k}}(m_{k}) \to \phi\circ H_{s}(m)$, since $s > 0$ 
and so $H_{s}(m) \in M\setminus X$.

Thus in particular $(M\setminus X) \cup C_{+}(\phi,H)$ is homeomorphic 
to $MC(p)\setminus ((M\setminus X) \times (N\setminus Y))$. The second 
homeomorphism now becomes obvious, as 
$C_{-}(\phi) \cup (N\setminus Y)$ is already precisely $MC(q)$ (note 
that $m_{k} \to m \in M\setminus X$ and $s_{k} \to 0$ imply 
$\phi\circ H_{s_{k}}(m_{k}) \to \phi(m)$ for any choice of $H$).

(2)The closed part is clear. Now take some homotopy 
$G \colon N \times [0,1] \to N$ with $G_{0} = id_{N}$ and 
$G(N \times (0,1]) \subseteq N\setminus Y$. We shall define the required 
homotopy $F\colon AC(\phi,H) \times [0,1] \to AC(\phi,H)$ first on 
$C_{-}(\phi) \cup C_{-,0}(\phi) \cup N$ by 
$F_{t}(m,s,n) := (H_{t}(m),s,G_{t}(n))$. 

Now define $F$ on $M \cup C_{+}(\phi,H) \cup C_{+,0}(\phi,H)$ by 
$$
F_{t}(m,s,\phi\circ H_{s}(m)) := 
\begin{cases}
  (H_{t}(m),s-t,\phi\circ H_{s-t}\circ H_{t}(m)), 
  & t < s \\
  (H_{t-s}\circ H_{s}(m),0,G_{t-s}\circ\phi\circ H_{s}(m)), 
  & t \geq s
\end{cases}.
$$
Then $F_{0} = id_{AC(\phi,H)}$ and 
$F(AC(\phi,H) \times (0,1]) \cap FC(\phi,H) = \varnothing$, as 
required. Moreover, $F$ is fully continuous: the only complicated case 
is when a sequence $\{(m_{k},s_{k},\phi\circ H_{s_{k}}(m_{k}))\}$ of 
points of $M \cup C_{+}(\phi,H) \cup C_{+,0}(\phi,H)$ 
converges to a point $(x,0,y) \in C_{-,0}(\phi)$. In this case, we have 
$m_{k} \to x$, $s_{k} \to 0$, and $\phi\circ H_{s_{k}}(m_{k}) \to y$. 
Now for a sequence $t_{k} \to t$ in $[0,1]$, consider two separate 
cases: $t_{k} < s_{k}$ and $t_{k} \geq s_{k}$; any sequence can clearly 
be split into two (at most) sequences with either satisfying one of 
these conditions for all indices.

In the case $t_{k} \geq s_{k}$, we have 
$H_{s_{k}}(m_{k}) \to x = H_{0}(x)$ 
(as $s_{k} \to 0$, $m_{k} \to x$, and $H$ is continuous), so 
$H_{t_{k}-s_{k}}\circ H_{s_{k}}(m_{k}) = H_{t_{k}}(m_{k}) \to H_{t}(x)$, 
because $H$ is additive; and also from 
$\phi\circ H_{s_{k}}(m_{k}) \to y$ we get 
$G_{t_{k}-s_{k}}\circ\phi\circ H_{s_{k}}(m_{k}) \to G_{t}(y)$ (again, as 
$s_{k} \to 0)$, thus 
$F_{t_{k}}(m_{k},s_{k},\phi\circ H_{s_{k}}(m_{k})) = 
(H_{t_{k}-s_{k}}\circ H_{s_{k}}(m_{k}),0,G_{t_{k}-s_{k}}\circ\phi\circ H_{s_{k}}(m_{k})) \to (H_{t}(x),0,G_{t}(y))$.

In the case $t_{k} < s_{k}$, we must have $t_{k} \to 0$, and therefore
$\phi\circ H_{s_{k}-t_{k}}\circ H_{t_{k}}(m_{k}) = \phi\circ H_{s_{k}}(m_{k}) \to y = G_{0}(y)$, 
since $H$ is additive. Thus 
$F_{t_{k}}(m_{k},s_{k},\phi\circ H_{s_{k}}(m_{k})) = 
(H_{t_{k}}(m_{k}),s_{k} - t_{k},\phi\circ H_{s_{k}-t_{t}}\circ H_{t_{k}}(m_{k})) \to (H_{0}(x),0,G_{0}(y)) = (x,0,y)$, 
as required.

Finally, if $Y$ is additive homotopy negligible in $N$, then $G$ can be 
chosen additive. Now $F$ becomes additive, as $H$ and $G$ are.

(3)Let $H^{(1)}$ and $H^{(2)}$ be two possible choices of homotopy, and 
denote $AC_{i} := AC(\phi,H^{(i)})$ and $FC_{i} := FC(\phi,H^{(i)})$ for 
$i = 1,2$. By (2), the pairs $(AC_{1},FC_{1})$ and $(AC_{2},FC_{2})$ are 
objects of $\mathrm{MAppr}$; by (1), there is a homeomorphism between 
$AC_{1}\setminus FC_{1}$ and $AC_{2}\setminus FC_{2}$. This 
homeomorphism is $FC_{1}-FC_{2}$-approaching: it extends by two separate 
homeomorphisms onto $X \cup C_{-,0}$ and onto $C_{+,0} \cup Y$, so the 
only case we need to check is that of a sequence 
$(m_{k},s_{k},H^{(1)}_{s_{k}}(m_{k}))$ in $C_{+}(\phi,H^{(1)})$ 
converging to a point $(x,0,y) \in C_{-,0}(\phi)$. But in this case 
$m_{k} \to x$ and $s_{k} \to 0$, so also $H^{(2)}_{s_{k}}(m_{k}) \to x$, 
and then $\phi\circ H^{(2)}_{s_{k}}(m_{k})$ must have some accumulation 
point $y' \in Y$. Thus $(x,0,y')$ is a accumulation point of 
$(m_{k},s_{k},H^{(2)}_{s_{k}}(m_{k}))$, proving the homeomorphism 
approaching. The inverse homeomorphism is similarly 
$FC_{2}-FC_{1}$-approaching, so we have an isomorphism of 
$(AC_{1},FC_{1})$ and $(AC_{2},FC_{2})$ as objects of $\mathrm{MAppr}$.

(4)As $FC(\phi,H)$ is a homotopy negligible subset by (2), it suffices 
to prove that $AC(\phi,H)\setminus FC(\phi,H)$ is an AR, as per 
Proposition~\ref{AR_props}(7). But by (1), 
$AC(\phi,H)\setminus FC(\phi,H)$ is an intersection of two metrizable 
mapping cylinders between ARs along an AR (the graph of $\phi$ in 
$(M\setminus X) \times \{0\} \times (N\setminus Y) \subset M \star N$, 
which is homeomorphic to $M\setminus X$) that is a closed subset of 
either. Also given any choices of $H^{(1)}$ and $H^{(2)}$, the 
approaching homeomorphisms of (3) now provide a fine shape 
isomorphism between $FC(\phi,H^{(1)})$ and $FC(\phi,H^{(2)})$.
\end{proof}

It follows from the proposition that all versions of 
$FC(\phi)$ for various $H$ represent the same isomorphism class of 
objects of $\mathrm{fSh(M)}$, and all versions of the pair 
$(AC(\phi),FC(\phi))$ represent the same isomorphism class of objects of 
$\mathrm{MAppr}$. Moreover, the fine shape class of any specific 
approaching map from $AC(\phi)\setminus FC(\phi)$, or into it, is also
independent of the choice of $H$. Thus we are justified in adopting the 
following

\begin{convention}
From now on, we shall refer to the two cylinders as $AC(\phi)$ and 
$FC(\phi)$, understanding that $H$ is chosen arbitrarily (as long as it 
is additive). Whenever needed, an arbitrary point of 
$C_{+}(\phi,H) \cup C_{+,0}(\phi,H)$ shall be denoted by 
$(m,s,\phi\circ H_{s}(m))$ for some $m \in M$ and $s \in (0,1)$, and a 
point of $C_{-}(\phi) \cup C_{-,0}(\phi)$ by $(m,s,n)$, where $m \in M$, 
$s \in (-1,0]$, and $n \in N$.
\end{convention}

The first indication that these cylinders allow us to invert the $S$ 
functor of section~\ref{sec_functor_S} is given by

\begin{proposition}\label{FC_fdr}
Given an approaching map $\phi\colon M\setminus X \to N\setminus Y$, if 
$M$ and $N$ are ARs, and $X$ is additive homotopy negligible in $M$, 
denote the standard embeddings as $u\colon X \to FC(\phi)$ and 
$i\colon Y \to FC(\phi)$. Then $i$ is an FDR-embedding, and 
$[i]_{fSh}^{-1}\circ[u]_{fSh} = [\phi]$.
\end{proposition}

\begin{corollary}\label{fSh_represent}
For any fine shape class $[\phi] \in [X,Y]_{fSh}$, there exist a space 
$Z$ along with maps (which can be chosen to be closed embeddings) 
$u\colon X \to Z$ and $i\colon Y \to Z$ such that $i$ is a fine shape 
equivalence and $[i]_{fSh}^{-1}\circ[u]_{fSh} = [\phi]$.
\end{corollary}

\begin{proof}[Proof (of Proposition~\ref{FC_fdr})]
$u$ and $i$ readily extend to embeddings $\bar{u}\colon M \to AC(\phi)$ 
and $\bar{i}\colon N \subset AC(\phi)$, and $FC(\phi)$ is closed and 
homotopy negligible in $AC(\phi)$ by the point 2) of the previous 
proposition. The approaching strong deformation retraction of 
$AC(\phi)\setminus FC(\phi)$ onto $N\setminus Y$ can be constructed in 
two steps. First, strong deformation retract 
$AC(\phi)\setminus FC(\phi)$ onto $C_{-}(\phi) \cup (N\setminus Y)$ by 
retracting each interval $\{(m,s,\phi\circ H_{s}(m)) \mid s \in [0,1]\}$ 
along itself; given the point 1) of the previous preposition, this is 
the same as retracting the metrizable cylinder of the map 
$$M\setminus X \ni m \mapsto (m,\phi(m)) \in (M\setminus X) \times (N\setminus Y).$$ 
Second, strong deformation retract $C_{-}(\phi) \cup (N\setminus Y)$ 
onto $N\setminus Y$, treating it as the metrizable cylinder of the map 
$$(M\setminus X) \times (N\setminus Y) \ni (m,n) \mapsto n \in (N\setminus Y).$$ 
Both these retractions are clearly correctly approaching, noting yet 
again that if 
$(m_{k},s_{k},\phi\circ H_{s_{k}}(m_{k})) \to (x,0,y) \in C_{-,0}(\phi)$, 
then $\phi\circ H_{t_{k}}(m_{k})$ has an accumulation point in $Y$ for 
any sequence $t_{k} \to 0$.

Now for any point $m \in M$, we have $\bar{u}(m) = (m,1)$, whose image 
at the end of this approaching strong deformation retraction is 
$\phi(m)$. Thus the class $[i]_{fSh}^{-1}\circ[u]_{fSh}$ is represented 
by $\phi$.
\end{proof}

\begin{remark}\label{FC_equiv_to_MC}
Consider a fine shape class induced by a map $f\colon X \to Y$. For any 
ARs $M$ and $N$ containing $X$ and $Y$ respectively as closed homotopy 
negligible subsets (additive or not), and any extension 
$\bar{f}\colon M \to N$ with $\bar{f}^{-1}(Y) = X$, the metrizable 
mapping cylinder $MC(f)$ is closed and homotopy negligible in 
$MC(\bar{f})$ (by~\cite[Proposition~19.8(a)]{Melikhov2022T}), which is 
an AR (by Proposition~\ref{AR_props}(3)); thus $\bar{f}|_{M\setminus X}$ 
represents $[f]_{fSh}$. Since $MC(\bar{f})$ strong deformation retracts 
onto $N$, restricting to a strong deformation retraction of $MC(f)$ onto 
$Y$, we conclude that $MC(f)$ is (in particular) fine shape isomorphic 
to $Y$, and therefore also to $FC(\bar{f}|_{M\setminus X})$. Thus our 
cylinder is consistent, at least in fine shape, with the usual 
metrizable cylinder.
\end{remark}

Now we state the first result related to the fine shape mapping cylinder 
and extension set of an approaching map 
(see Definition~\ref{def_phi_ex_set}).

\begin{proposition}\label{embed_MC}
Let $X$ be closed additively homotopy negligible in $M$, and let 
$\phi\colon M\setminus X \to N\setminus Y$ be an approaching map 
that extends on $A \subseteq X$ by a map $f\colon A \to Y$. Then there 
is an embedding of the metrizable mapping cylinder $MC_{X}(f)$ into 
$FC(\phi)$; therefore, if $u\colon X \subset FC(\phi)$ and 
$i\colon Y \subset FC(\phi)$ are the standard embeddings, then 
$if \simeq u|_{A}$.
\end{proposition}

\begin{proof}
Given any choice of homotopy $H$, we embed 
$MC(f) = X \times \{1\} \cup A \times (0,1) \cup N$ by sending a 
point $y \in Y$ or $x \in X$ to itself (through the standard 
embeddings), while sending a point $(a,s) \in A \times (0,1)$ to 
$(a,2s-1,\phi\circ H_{2s-1}(a))$ for $s \in (\frac{1}{2},1]$ or to 
$(a,2s-1,f(a))$ for $s \in (0,\frac{1}{2}]$. Continuity of this mapping 
is assured by the fact that $A \subseteq X(\phi)$. 
Specifically, denote the combination of $\phi$ and $f$ by $\bar{\phi}$. 
Whenever $m_{k} \to m$ in $(M\setminus X) \cup A$, we have 
$\bar{\phi}(m_{k}) \to \bar{\phi}(m)$, and if also $s_{k} \to s \geq 0$, 
then $H_{s_{k}}(m_{k}) \to H_{s}(m)$, so 
$\bar{\phi}\circ H_{s_{k}}(m_{k}) \to \bar{\phi}\circ H_{s}(m)$.
\end{proof}

The last proposition of this section describes the primary property of 
the pair $(AC(\phi),FC(\phi))$ that makes it similar to the usual 
mapping cylinder. In addition, it also includes a result about extension 
sets. This, combined with the previous proposition, will be used to 
prove fraction equivalences in the next section.

\begin{proposition}\label{AC_extend}
Let $X$ be closed additively homotopy negligible in $M$. Assume an 
approaching map $\phi\colon M\setminus X \to N\setminus Y$. Take 
$(AC(\phi),FC(\phi))$, and denote the embeddings of $M$ and $N$ by 
$\bar{u}$ and $\bar{\i}$ respectively. Given approaching maps 
$\psi\colon M\setminus X \to L\setminus Z$ and 
$\chi\colon N\setminus Y \to L\setminus Z$, along with an approaching 
homotopy 
$\Psi\colon (M\times [0,1])\setminus (X\times [0,1]) \to L\setminus Z$ 
with $\Psi_{1} = \psi$ and $\Psi_{0} = \chi\circ\phi$, there is an 
approaching map
$\omega \colon AC(\phi)\setminus FC(\phi) \to L\setminus Z$ such 
that $\omega\circ\bar{u}|_{M\setminus X} = \psi$ and 
$\omega\circ\bar{\i}|_{N\setminus Y} = \chi$. 
Moreover, $\omega$ can be chosen so that it extends on $X(\psi)$ and 
$Y(\chi)$ by the same maps that extend $\psi$ and $\chi$.
\end{proposition}

\begin{proof}
Given any choice of additive homotopy $H$, define $\omega$ by
$$
\omega|_{(M\setminus X) \cup C_{+}(\phi,H)}(m,s,\phi\circ H_{s}(m)) := 
\begin{cases}
\psi(m), & s \in [\frac{1}{2},1] \\
\Psi_{2s}(m), & s \in (0,\frac{1}{2}]
\end{cases}
$$
and 
$\omega|_{C_{-}(\phi) \cup (N\setminus Y)}(m,s,n) := \chi(n)$. All 
required properties are trivial to prove.
\end{proof}

\begin{remark}\label{AC_universal}
In the topological category, the usual mapping cylinder of a map 
$f\colon X \to Y$ is the pushout of the diagram 
$X \times [0,1] \xleftarrow{x = (x,0)} X \xrightarrow{f} Y$; the same is 
true for the metrizable mapping cylinder in $\mathrm{M}$ (with respect 
to all metrizable spaces, but not to all topological ones), which had to 
be defined, to begin with, because the topological pushout of metrizable 
spaces may not be metrizable. The situation with $\mathrm{MAppr}$ seems 
to be even worse and less studied too; our proposition, therefore, does 
not even show the existence part of the same universal property (and our 
cylinder pair, while definable even with $H$ not additive, may not be 
an object of $\mathrm{MAppr}$ in that case). It is not immediately clear 
whether $\mathrm{MAppr}$ has pushouts even of this restricted type (see 
the appendix; the counterexample is applicable here too). What we do 
prove suffices for our purposes, but there may be some interest in 
studying the matter further.
\end{remark}

\section{The T functor}
\label{sec_functor_T}

Now we are in position to define the inverse to $S$. First we provide a 
lemma that uses Propositions~\ref{AC_extend} and~\ref{embed_MC} to prove 
fraction equivalence, and will be used in proving the next two theorems.

\begin{lemma}\label{fraction_equiv}
Let $\phi\colon M\setminus X \to N\setminus Y$ be an approaching map, 
$u\colon X \to FC(\phi)$ and $i\colon Y \to FC(\phi)$ be the 
standard embeddings, and 
$(M,X) \xrightarrow{\bar{v}} (L,Z) \xleftarrow{\bar{\j}} (N,Y)$ be a 
cospan of maps of pairs with $\bar{v}^{-1}(Z) = X$ and 
$\bar{\j}^{-1}(Z) = Y$; here $M$, $N$, and $L$ are ARs, and $X$, $Y$, 
and $Z$ are their closed homotopy negligible subsets, with $X$ and $Y$ 
additively so. Denote $v := \bar{v}|_{X}\colon X \to Z$ and 
$j := \bar{\j}|_{Y}\colon Y \to Z$. If the homotopy class $[j]$ is in 
$[FDR]$, and if there is an approaching homotopy 
$\Omega\colon (M \times [0,1])\setminus(X \times [0,1]) \to L\setminus Z$ 
between $\Omega_{0} = \bar{v}|_{M\setminus X}$ and 
$\Omega_{1} = \bar{\j}|_{N\setminus Y}\circ\phi$, then the left 
fractions $X \xrightarrow{[v]} Z \xleftarrow{[j]} Y$ and 
$X \xrightarrow{[u]} FC(\phi) \xleftarrow{[i]} Y$ are equivalent.
\end{lemma}

\begin{proof}
To prove the fractions equivalent, we need to exhibit a 
cospan $FC(\phi) \xrightarrow{[f]} P \xleftarrow{[g]} Z$ such that 
$[fu] = [gv]$, $[fi] = [gj]$, and the latter class is in $[FDR]$. So 
use Proposition~\ref{AC_extend} to obtain an approaching map 
$\omega\colon AC(\phi)\setminus FC(\phi) \to L\setminus Z$ that 
extends on $X$ and $Y$ by $v$ and $j$ respectively. Now take 
$P$ to be $FC(\omega)$ (note that by Proposition~\ref{AC_FC_props}(2), 
$FC(\phi)$ is additive homotopy negligible in $AC(\phi)$ since $Y$ is so 
in $N$; thus $(AC(\omega),FC(\omega))$ is defined), with $f$ and $g$ the 
standard embeddings:

$$
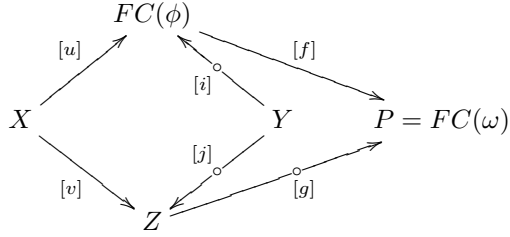
\xymatrix{
 & FC(\phi) \ar[drr]^{[f]} & & \\
X \ar[ur]^{[u]} \ar[dr]_{[v]} & & Y \ar[ul]|{\circ}^{[i]} \ar[dl]|{\circ}_{[j]} & P = FC(\omega) \\
 & Z \ar[urr]|{\circ}_{[g]} & & \\
}$$

\noindent Then by Proposition~\ref{embed_MC}, $P$ contains 
embeddings of $MC(v)$ and $MC(j)$ that prove $[fu] = [gv]$ and 
$[fi] = [gj]$, and $[gj]$ is in $[FDR]$ as $[j]$ and $[g]$ are.
\end{proof}

\begin{remark}\label{fraction_equiv_remark}
The only reason we require $[j] \in [FDR]$ is that otherwise the 
span $X \xrightarrow{[v]} Z \xleftarrow{[j]} Y$ does not define a 
fraction class in $[FDR]\backslash hM$. On the other hand, the 
approaching homotopy $\Omega$ may in some cases be provided by composing 
$\bar{v}|_{M\setminus X}$ with the approaching strong deformation 
retraction of $L\setminus Z$ onto $N\setminus Y$. The proof of 
Theorem~\ref{ST_inverse} below uses exactly this.

Also note that $Z$ does not have to be {\it additive} homotopy 
negligible in $L$.
\end{remark}

With this lemma we can prove

\begin{theorem}\label{define_T}
There is a functor 
$T \colon \mathrm{fSh(M)} \to [FDR]\backslash\mathrm{hM}$, uniquely 
defined by the following properties:

(1)$T$ is constant on objects;

(2)Whenever a fine shape class $[\phi] \in [X,Y]_{fSh}$ can be 
represented by an $X-Y$-approaching map 
$\phi \colon M\setminus X \to N\setminus X$, where $M$ and $N$ are ARs 
containing $X$ and $Y$ as closed additive homotopy negligible subsets, 
the left fraction class $T([\phi])$ can be represented by the fraction 
$X \xrightarrow{[u]} FC(\phi) \xleftarrow{[i]} Y$, where 
$[u]$ and $[i]$ are the homotopy classes of the standard embeddings.
\end{theorem}

\begin{proof}
The properties define the action of $T$ uniquely, and the left fraction 
specified in (2) is well-defined, as $i\colon Y \to FC(\phi)$ is an 
FDR-embedding by Proposition~\ref{FC_fdr}. Thus we only need to check 
that $T$ is a functor.

First, we show that for any two approaching maps representing the same 
fine shape class, the left fractions specified in (2) are equivalent. To 
that end, let a given fine shape class from $X$ to $Y$ be represented by 
both approaching maps $\phi\colon M\setminus X \to N\setminus Y$ and 
$\psi\colon M'\setminus X \to N\setminus Y'$, where $M$ and $M'$ are ARs 
each containing $X$ as a closed additive homotopy negligible subset, and 
$N$ and $N'$ are ARs containing $Y$ as such. This means (by 
Definition~\ref{def_fSh_class}) that for some approaching maps 
$\bar{id}_{X}\colon M\setminus X \to M'\setminus X$ and 
$\bar{id}_{Y}\colon N\setminus Y \to N'\setminus Y$ extending $id_{X}$ 
and $id_{Y}$ we have an approaching homotopy 
$\Psi \colon (M\setminus X)\times [0,1] \to N'\setminus Y$ 
between $\Psi_{1} = \psi\circ\bar{id}_{X}$ and 
$\Psi_{0} = \bar{id}_{Y}\circ\phi$. Let 
$X \xrightarrow{[u]} FC(\phi) \xleftarrow{[i]} Y$ and 
$X \xrightarrow{[v]} FC(\psi) \xleftarrow{[j]} Y$ be the fractions 
consisting of the homotopy classes of the embeddings (the embeddings 
themselves shall be denoted by $u,v,i,j$ accordingly). We want to prove 
them equivalent. 

The maps $u,v,i,j$ extend to embeddings
$M \xrightarrow{\bar{u}} AC(\phi) \xleftarrow{\bar{\i}} N$ and 
$M' \xrightarrow{\bar{v}} AC(\psi) \xleftarrow{\bar{\j}} N'$. 
We have approaching maps 
$\bar{v}\circ\bar{id}_{X} \colon M\setminus X \to AC(\psi)\setminus FC(\psi)$ 
and 
$\bar{\j}\circ\bar{id}_{Y} \colon N\setminus Y \to AC(\psi)\setminus FC(\psi)$. 
Assuming $AC(\psi) = AC(\psi,H)$ for some specific homotopy $H$, define 
$\Omega \colon (M\setminus X) \times [0,1] \to AC(\psi)\setminus FC(\psi)$ 
by 
$$
\Omega(m,s) := 
\begin{cases}
(\bar{v}\circ\bar{id}_{X}(m),4s-3,\psi\circ H_{4s-3}\circ\bar{v}\circ\bar{id}_{X}(m)), & s \in [\frac{3}{4},1] \\
(\bar{v}\circ\bar{id}_{X}(m),4s-3,\psi\circ\bar{v}\circ\bar{id}_{X}(m)), & s \in [\frac{1}{2},\frac{3}{4}] \\
\Psi_{2s}(m) \in (N'\setminus Y) \subset AC(\psi)\setminus FC(\psi), & s \in [0,\frac{1}{2}]
\end{cases}
$$
As defined, $\Omega$ is an approaching 
homotopy between $\Omega_{1} = \bar{v}\circ\bar{id}_{X}$ and 
$\Omega_{0} = \bar{\j}\circ\bar{id}_{Y}\circ\phi$. Then by 
Lemma~\ref{fraction_equiv}, $[i]\backslash [u]$ and $[j]\backslash [v]$ 
are equivalent.

Second, it is clear that an identity fine shape class $[id_{X}]_{fSh}$ 
corresponds to the identity fraction class: represent $[id_{X}]_{fSh}$ 
by $id_{M}$ for some AR $M$ containing $X$ as a closed additive homotopy 
negligible subset, and note that by Proposition~\ref{embed_MC}, 
$FC(id_{M})$, contains a copy of 
$X \times [0,1] = MC(id_{X})$, which makes the fraction 
$X \to FC(id_{M}) \leftarrow X$ equivalent to 
$X \to X \times [0,1] \leftarrow X$, which belongs to the identity 
fraction class.

Third, consider the composition. Assume fine shape classes 
$[\phi] \in [X,Y]_{fSh}$ and $[\psi] \in [Y,Z]_{fSh}$ are represented by 
approaching maps $\phi\colon M\setminus X \to N\setminus Y$ and 
$\psi\colon N\setminus Y \to L\setminus Z$ (where $M$, $N$, and $L$ are 
ARs containing $X$, $Y$, and $Z$ as closed additive homotopy negligible 
subset), and consider the fractions. We have maps 
$X \xrightarrow{u} FC(\phi) \xleftarrow{i} Y$, 
$Y \xrightarrow{v} FC(\psi) \xleftarrow{j} Z$, and 
$X \xrightarrow{w} FC(\psi\circ\phi) \xleftarrow{k} Z$, which all extend 
to embeddings 
$M \xrightarrow{\bar{u}} AC(\phi) \xleftarrow{\bar{\i}} N$, 
$N \xrightarrow{\bar{v}} AC(\psi) \xleftarrow{\bar{\j}} L$, and 
$M \xrightarrow{\bar{w}} AC(\psi\circ\phi) \xleftarrow{\bar{k}} L$. 
For the composition of fraction classes 
$[j]\backslash [v]\circ[i]\backslash [u]$, we can take an approaching map 
$\chi \colon AC(\phi)\setminus FC(\phi) \to AC(\psi)\setminus FC(\psi)$ 
that simply collapses $AC(\phi)\setminus FC(\phi)$ onto 
$N\setminus Y$, which is then embedded into 
$AC(\psi)\setminus FC(\psi)$; clearly $\chi$ extends on $Y$ by 
identity. Take $FC(\chi)$ ($FC(\phi)$ is additive homotopy negligible in 
$AC(\phi)$ since $Y$ is so in $N$, per Proposition~\ref{AC_FC_props}(2); 
thus $FC(\chi)$ is well-defined), and denote the standard embeddings 
by $FC(\phi) \xrightarrow{f} FC(\chi) \xleftarrow{g} FC(\psi)$, 
extending to 
$AC(\phi) \xrightarrow{\bar{f}} AC(\chi) \xleftarrow{\bar{g}} AC(\psi)$. 
Then $[gv] = [fi]$ (homotopy 
provided by $\chi$ extending on $Y$ and Proposition~\ref{embed_MC}), and 
$g$ is an FDR-embedding, thus the fraction $[gj]\backslash [fu]$ belongs 
to the composition class $([j]\backslash [v])\circ([i]\backslash [u])$:

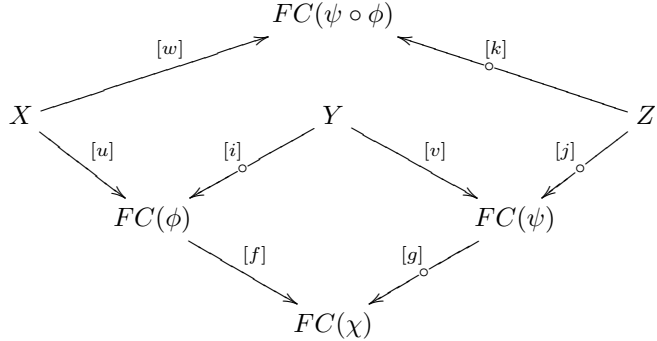
$$\xymatrix{
 & & FC(\psi\circ\phi) & & \\
X \ar[dr]^{[u]} \ar[urr]^{[w]} & & Y \ar[dl]|{\circ}_{[i]} \ar[dr]^{[v]} & & Z \ar[dl]|{\circ}_{[j]} \ar[ull]|{\circ}_{[k]} \\
 & FC(\phi) \ar[dr]^{[f]} & & FC(\psi) \ar[dl]|{\circ}_{[g]} & \\
 & & FC(\chi) & & 
}$$

We have the embeddings $\bar{f}\circ\bar{u} \colon M \to AC(\chi)$ 
and $\bar{g}\circ\bar{\j} \colon L \to AC(\chi)$, as well as an 
approaching homotopy 
$\Omega \colon (M\setminus X) \times [0,1] \to AC(\chi)\setminus FC(\chi)$ 
with $\Omega_{1} = \bar{f}\circ\bar{u}|_{M\setminus X}$ and 
$\Omega_{0} = \bar{g}\circ\bar{\j}\circ\psi\circ\phi$ 
(provided by approaching strong deformation retracting 
$AC(\chi)\setminus FC(\chi)$ onto 
$AC(\psi)\setminus FC(\psi)$ and then onto $L\setminus Z$). Thus 
by Lemma~\ref{fraction_equiv} we obtain the equivalence 
$T([\psi]\circ[\phi]) = [k]\backslash [w] = [gj]\backslash [fu] = 
([j]\backslash [v])\circ([i]\backslash [u]) = T([\psi])\circ T([\phi])$.
\end{proof}

From this proof, we also derive

\begin{corollary}\label{maps_fSh_homotopic}
Given two maps $f,g\colon X \to Y$ such that $[f]_{fSh} = [g]_{fSh}$, 
there exist a space $Z$ along with a map $h\colon Y \to Z$ such that $h$ 
is a closed embedding and a fine shape equivalence, and $hf \simeq hg$. 
Moreover, there are closed embeddings of $MC(f)$ and $MC(g)$ into $Z$, 
at least one of which is a fine shape equivalence.
\end{corollary}

\begin{proof}
Choose any ARs $M$ and $N$ containing $X$ and $Y$ respectively as closed 
homotopy negligible subsets, and choose respective extensions 
$\bar{f},\bar{g}\colon M \to N$ with 
$\bar{f}^{-1}(Y) = \bar{g}^{-1}(Y) = X$ (of course, $\bar{f}|_{X} = f$ 
and $\bar{g}|_{X} = g$). As per Remark~\ref{FC_equiv_to_MC}, we have 
well-defined left fractions $X \to MC(f) \xleftarrow{} Y$ and 
$X \to MC(g) \xleftarrow{} Y$ representing $T([f]_{fSh})$ and 
$T([g]_{fSh})$ (additive homotopy negligibility rendered irrelevant by 
the same remark); since $[f]_{fSh} = [g]_{fSh}$, these fractions must be 
equivalent. The corollary now follows from the definition of fraction 
equivalence (Definition~\ref{def_left_cat}(3)).

The more or less explicit construction, as derived from our proofs, is 
as follows: first we take some approaching homotopy, say, 
$F\colon (M\setminus X) \times [0,1] \to N\setminus Y$ between 
$F_{0} = \bar{f}|_{M\setminus X}$ and $F_{1} = \bar{g}|_{M\setminus X}$ 
(which affirms the equality $[f]_{fSh} = [g]_{fSh}$ as per 
Definition~\ref{def_fSh_class}). Second, we construct
$\omega\colon MC(\bar{f})\setminus MC(f) \to MC(\bar{g})\setminus MC(g)$ 
as in the proof of Proposition~\ref{AC_extend}, using $F$, such that 
$\omega$ extends onto the standard embeddings of $X$ and $Y$ into 
$MC(f)$, and carries those to the standard embeddings of the same into 
$MC(g)$. Third, as in the proof of Lemma~\ref{fraction_equiv}, we take 
$Z = FC(\omega)$; by the construction of $FC(\omega)$, $Z$ contains 
embeddings of $MC(f)$ and $MC(g)$, the latter being an FDR-embedding
(Proposition~\ref{FC_fdr}), but also $Z$ contains 
(Proposition~\ref{embed_MC}) embeddings of $X \times [0,1]$ and 
$Y \times [0,1]$ which are precisely the mapping cylinders of the 
extensions of $\omega$ onto $X$ and $Y$ respectively. Fourth, consider 
the embedded thus ``rectangular'' subspace 
$MC(f) \cup (X \times [0,1]) \cup MC(g) \cup (Y \times [0,1])$ of $Z$; 
here the union is such that $X \times \{0\}$ and $Y \times \{0\}$ are 
the standard embeddings into $MC(f)$, whereas $X \times \{1\}$ and 
$Y \times \{1\}$ are those into $MC(g)$. We see that the closed 
embeddings $h_{0}\colon Y = Y \times \{0\} \subset MC(f) \subset Z$ and 
$h_{1}\colon Y = Y \times \{1\} \subset MC(g) \subset Z$ are homotopic 
(along $Y \times [0,1])$, and $h_{1}$ is an FDR-embedding as a 
composition of such; choose $h := h_{1}$. Now the embeddings 
$X = X \times \{0\} \subset MC(f) \subset Z$ and 
$X = X \times \{1\} \subset MC(g) \subset Z$ are homotopic to each 
other (along $X \times [0,1]$), but also to $h_{0}f$ and $h_{1}g$ (along 
$MC(f)$ and $MC(g)$) respectively, thus 
$hf = h_{1}f \simeq h_{0}f \simeq h_{1}g = hg$.
\end{proof}

Finally we obtain

\begin{theorem}\label{ST_inverse}
The functors $S$ and $T$ are inverse to each other: 
$ST = id_{\mathrm{fSh(M)}}$, $TS = id_{[FDR]\backslash \mathrm{hM}}$, 
providing a category isomorphism 
$\mathrm{fSh(M)} \cong [FDR]\backslash \mathrm{hM}$.
\end{theorem}

\begin{proof}
First, both functors are constant on objects.

Second, for any fine shape class $[\phi]$ we have 
$ST([\phi]) = [\phi]$ from the definitions and Proposition~\ref{FC_fdr}.

Third, let a left fraction class from $X$ to $Y$ be represented by 
specific maps $X \xrightarrow{u} Z \xleftarrow{i} Y$ with $i$ an 
FDR-map. The latter means there exist ARs $L \supset Z$ and 
$N \supset Y$, where inclusions are those of closed homotopy negligible 
subsets with $Y$ additively so (use Proposition~\ref{AR_props}(7) and 
Corollary~\ref{fdr_AR_property}; see also 
Remark~\ref{fraction_equiv_remark}), the embedding 
$\bar{\i}\colon N \subseteq L$ with $N \cap Z = Y$ extending $i$, and 
the $Z\times [0,1] - Z$-approaching strong deformation retraction 
$\Psi\colon (L\setminus Z)\times [0,1] \to L\setminus Z$ such that 
$\Psi_{1}$ is a specific $Z-Y$-approaching map from $L\setminus Z$ to 
$N\setminus Y$ representing the fine shape class $[i]_{fSh}^{-1}$. Now 
let $M$ be any AR containing $X$ as a closed additive homotopy 
negligible subset, and $\phi\colon M\setminus X \to L\setminus Z$ any 
specific approaching map extending $u$, and thus representing 
$[u]_{fSh}$. Then $S([i]\backslash [u])$ can be represented by 
$\Psi_{1} \circ \phi$. Take $FC(\Psi_{1}\circ\phi)$, and denote the 
embeddings by 
$X \xrightarrow{v} FC(\Psi_{1}\circ\phi) \xleftarrow{j} Y$. The map 
$$(M\setminus X) \times [0,1] \ni (m,t) \mapsto \Psi_{t}\circ\phi(m) \in L\setminus Z$$
provides an approaching homotopy between 
$\phi = \Psi_{0}\circ\phi$ and $\Psi_{1}\circ\phi$. By 
Lemma~\ref{fraction_equiv}, we obtain the desired equivalence
$[j]\backslash [v] = [i]\backslash [u]$, and conclude that 
$TS = id_{[FDR]\backslash hM}$.
\end{proof}

\appendix

\section*{Appendix}

Here we show why the approaching map cylinder, as defined by 
Mrozik~\cite{Mrozik1990} for compact spaces, does not seem to extend 
directly to the noncompact case.

Let an approaching map $\phi\colon M\setminus X \to N\setminus Y$ be 
given, and choose a homotopy $H\colon M \times [0,1] \to M$ with 
$H_{0} = id_{M}$ and $H(M \times (0,1]) \subseteq M\setminus X$; here 
$H$ being additive will not be relevant. We can easily define a 
metrizable topology on the set $M \times (0,1] \cup N$ by inserting it 
into $M \star N$: take the standard embedding of $N$, 
and send $(m,s) \in M \times (0,1]$ to $(m,2s-1,\phi\circ H_{s}(m))$. 
In terms of Definition~\ref{define_AC_FC}, this would extend 
$C_{+} \cup C_{+,0}$ across the whole join and remove 
$C_{-} \cup C_{-,0}$ entirely (or collapse it onto $N$). Similarly 
to~\cite{Mrozik1990}, denote the resulting space by $D(\phi)$, and write 
$C(\phi)$ for the subspace $X \times (0,1] \cup Y$.

The problem is that the obvious retraction given by 
$$
D(\phi)\setminus C(\phi) \ni (m,s) \mapsto \phi(m) \in N\setminus Y
$$
is not, in general, a $C(\phi)-Y$-approaching map, and therefore there 
is no fine shape class inverse to the one given by the embedding 
$Y \subseteq C(\phi)$. By the definition of a metrizable join, a 
sequence $(m_{k},2s_{k}-1,\phi\circ H_{s_{k}}(m_{k}))$ converges to a 
point $y \in Y$ if and only if $s_{k} \to 0$ and 
$\phi\circ H_{s_{k}}(m_{k}) \to y$ in $N$. For the retraction to be 
approaching, we would need $\phi(m_{k})$ to have an accumulation point 
in $Y$, which does not follow from the premises in the noncompact case. 
This also means that $(D(\phi),C(\phi))$ does not have the universal 
property of the usual cylinder (as mentioned in 
Remark~\ref{AC_universal}) outside of compacta.

For an actual counterexample, take the quotient set 
$M := \N \times [0,1]/(\N \times \{1\})$ in the metrizable topology 
inherited from a piecewise linear insertion of $M$ 
into $\R^{2}$ which inserts the apex $1$ as $(0,1)$, and $(k,0)$ as 
$(k,0)$ for all $k \in \N$. $M$ is an AR because it is a contractible 
polyhedron. Take $X := \N \times \{0\} \subset M$, 
$N := \R \times [0,1]$, and $Y := \R \times \{0\} \subset Y$. Define 
$\phi\colon M\setminus X \to N\setminus Y$ by 
$\phi(k,s) := (k\sin{\frac{\pi}{s}},s)$; this is well-defined for $s=1$. 
Then $\phi$ is $X-Y$-approaching. Given a homotopy $H$ on $M$ (as usual, 
$H_{0} = id_{M}$ and $H(M \times (0,1]) \subseteq M\setminus X$; $H$ can 
be chosen additive, but again it will not be relevant), we 
construct a sequence $((k,s_{k}),2t_{k}-1,\phi\circ H_{t_{k}}(k,s_{k}))$ 
in $M \star N$ such that $\phi\circ H_{t_{k}}(k,s_{k}) \to (0,0)$ in 
$N$, and $t_{k} \to 0$, but $\phi(k,s_{k})$ has no accumulation points 
in $Y$ (or even in $N$).

We construct ($s_{k}$,$t_{k}$) separately for each $k$. First, 
assume $U_{k}$ is the open set of $[0,1] \times [0,1]$ such that for all 
$(s,t) \in U$, $0 < s < \frac{1}{2^{k}}$, 
$0 \leq t < \frac{1}{2^{k}}$, and $H_{t}(k,s) \in \{k\} \times (0,1)$ 
($U_{k}$ is nonempty since $H_{0}(k,0) = (k,0)$ and $H$ is continuous). 
Then we can write $H_{t}(k,s) = (k,h_{k}(s,t))$ for some continuous 
function $h_{k}$ and all $(s,t) \in U_{k}$.

Consider the set $B_{k}$ of all pairs $(s,t) \in U_{k}$ such that 
$\sin{\frac{\pi}{s}} = 1$ and $\sin{\frac{\pi}{h_{k}(s,t)}}$ = 0. We 
claim that $B_{k}$ is nonempty: first, points 
$(s_{l},t) = (\frac{2}{4l+1},0)$, satisfying 
$\sin{\frac{\pi}{s_{l}}} = 1$, 
are in $U_{k}$ for all large natural $l$, and converge to $(0,0)$ in 
$[0,1] \times [0,1]$. Second, if for every (large) $l$ and for every 
$t \in [0,\frac{1}{2^{k}}]$ we were to have 
$\sin{\frac{\pi}{h_{k}(s_{l},t)}} \neq 0$, then also for every $t$, 
$|s_{l} - h_{k}(s_{l},t)| \to 0$. 
That implies $h_{k}(s_{l},t) \to 0$ as $l \to +\infty$, and then 
$H_{t}(k,0) = (k,0)$ for $t > 0$, contradicting the choice of $H$. 
Therefore, for some $l$ there is $t \in [0,\frac{1}{2^{k}}]$ such that 
$\sin{\frac{\pi}{h_{k}(s_{l},t)}} = 0$. Thus we can choose 
$(s_{k},t_{k}) \in B_{k}$. By construction, 
$s_{k} \to 0$, $t_{k} \to 0$, 
$\phi(k,s_{k}) = (k\sin{\frac{\pi}{s_{k}}},s_{k}) = (k,s_{k})$, and 
$\phi\circ H_{t_{k}}(k,s_{k}) = (0,h_{k}(s_{k},t_{k})) \to (0,0)$. This 
shows that in general $\phi(m_{k})$ having an accumulation point in $Y$ 
does not follow from $\phi\circ H_{t_{k}}(m_{k})$ having one.

This counterexample does not, of course, constitute definitive proof 
that there is no metrizable topology on $M \times (0,1] \cup N$ that 
makes the pair $(M \times (0,1] \cup N,X \times (0,1] \cup Y)$ into an 
object of $\mathrm{MAppr}$ having all the properties we need for 
Lemma~\ref{fraction_equiv}. The question of defining a mapping cylinder 
in $\mathrm{MAppr}$ is still an open one in the noncompact case, in the 
sense of both that lemma and the pushout property. For 
compact $M$ and $N$, it can be shown that $(D(\phi),C(\phi))$ 
defined above and $(AC(\phi),FC(\phi))$ are isomorphic objects of 
$\mathrm{hMAppr}$ by approaching strong deformation retracting the 
subset $C_{-}(\phi) \cup (N\setminus Y)$ of $AC(\phi)\setminus FC(\phi)$ 
onto its base $N\setminus Y$ --- in short, this shows that the cylinder 
of Mrozik is sufficient in the compact case.

\section*{Acknowledgements}
The author expresses gratitude to S.~Melikhov for the discussion of the 
subject and the results, and for remarks on theorem and proof 
formulations, as well as to the Advanced Doctoral Programme at Higher 
School of Economics for financial support.

\bibliographystyle{bibalph}
\bibliography{references} 
\end{document}